\documentclass[twoside,11pt]{article}

\usepackage{jmlr2e}

\usepackage[utf8]{inputenc} 
\usepackage[T1]{fontenc}    
\usepackage{hyperref}       
\usepackage{url}            
\hypersetup{colorlinks,linkcolor={blue},citecolor={blue},urlcolor={red}}  
\usepackage{booktabs}       
\usepackage{amsfonts}       
\usepackage{nicefrac}       
\usepackage{microtype}      %

\usepackage{color}
\usepackage{subcaption}

\let\OLDthebibliography\thebibliography
\renewcommand\thebibliography[1]{
  \OLDthebibliography{#1}
  \setlength{\parskip}{0pt}
  \setlength{\itemsep}{2pt plus 0.3ex}
}

\usepackage{graphicx}				
\usepackage{amssymb,amsmath}
\usepackage{enumitem}

\usepackage{algorithm}

\newcommand{\mc}[1]{\mathcal{#1}}
\newcommand*{\tran}{^{\mkern-1.5mu\mathsf{T}}}
\newtheorem{Assumption}{Assumption}

\newcommand{\rom}[1]{\uppercase\expandafter{\romannumeral #1\relax}}
\newcommand{\beas}{\begin{eqnarray*}}
\newcommand{\enas}{\end{eqnarray*}}
\newcommand{\bea}{\begin{eqnarray}}
\newcommand{\ena}{\end{eqnarray}}
\newcommand{\bms}{\begin{multline*}}
\newcommand{\ems}{\end{multline*}}
\newcommand{\bels}{\begin{align*}}
\newcommand{\enls}{\end{align*}}
\newcommand{\bel}{\begin{align}}
\newcommand{\enl}{\end{align}}

\newcommand{\ignore}[1]{}

\def\blfootnote{\xdef\@thefnmark{}\@footnotetext}

\newcommand{\expect}[1]{\mathbb{E}{\l[#1\r]}}

\def\mc{\mathcal}

\def\l{\left}
\def\r{\right}

\newcommand{\mb}{\mathbb}
\newcommand\argmin{\mathop{\mbox{argmin}}}

\newcommand{\mf}[1]{\mathbf{#1}}

\newcommand{\dotp}[2]{\left\langle#1,#2\right\rangle}

\newcommand\argmax{\mathop{\mbox{argmax}}}

\title{Online Primal-Dual Mirror Descent under Stochastic Constraints}

\begin{document}
\author{\name Xiaohan Wei \email xiaohanw@usc.edu \\
       \addr Department of Electrical Engineering\\
       University of Southern California\\
       Los Angeles, CA 90089, USA
       \AND
       \name Hao Yu \email hao.yu@alibaba-inc.com \\
       \addr Alibaba Group (U.S.) Inc.\\
       Bellevue, WA 98004, USA
       \AND
       \name Michael J. Neely \email mikejneely@gmail.com \\
       \addr Department of Electrical Engineering\\
       University of Southern California\\
       Los Angeles, CA 90089, USA\\
       }

\maketitle

\begin{abstract}
We consider online convex optimization with stochastic constraints where the objective functions are arbitrarily time-varying and the constraint functions are independent and identically distributed (i.i.d.) over time. Both the objective and constraint functions are revealed after the decision is made at each time slot. The best known expected regret for solving such a problem is 
$\mathcal{O}(\sqrt{T})$, with a coefficient that is polynomial in the dimension of the decision variable and relies on the \textit{Slater condition} (i.e. the existence of interior point assumption), which is restrictive and in particular precludes treating equality constraints. In this paper, we show that such Slater condition is in fact not needed. We
propose a new \textit{primal-dual mirror descent} algorithm and show that one can attain $\mathcal{O}(\sqrt{T})$ regret and constraint violation under a much weaker Lagrange multiplier assumption, allowing general equality constraints and significantly relaxing the previous Slater conditions. Along the way, for the 
case where decisions are contained in a probability simplex, we reduce the coefficient 
to have only a logarithmic dependence on the decision variable dimension. Such a dependence 
has long been known in the literature on mirror descent but seems
new in this new constrained online learning scenario.

\end{abstract}

\section{Introduction}
We consider an online convex optimization (OCO) problem with a sequence of arbitrarily varying convex objective functions $f^t(\mu),~t=0,1,2,\cdots,~\mu\in\Delta \subseteq \mathbb{R}^d$ which are revealed per slot after the decision is made, and $\Delta$ is a closed bounded convex set. For a fixed time horizon $T$, define the regret of a sequence of decisions 
$\l\{\mu^0,\mu^1,~\cdots,~\mu^{T-1}\r\}\subseteq\Delta$ as
\[
\sum_{t=0}^{T-1}f^t(\mu^t) - \min_{\mu\in\Delta }\sum_{t=0}^{T-1}f^t(\mu).
\] 
The goal of OCO is to choose the decision sequence so that the regret grows sublinearly with respect to $T$. OCO is a classical problem and has been considered in a number of previous works such as \citep{Cesa-Bianchi96TNN,Gordon99COLT,Zinkevich03ICML,Hazan16FoundationTrends}. In particular, it is known that for differentiable functions $f^t(\cdot)$, the projected gradient descent algorithm achieves an $\mathcal{O}(\sqrt{T})$ regret which is also worst case optimal. When the set $\Delta $ is a probability simplex, the mirror descent algorithm further achieves an ``almost dimension free'' logarithmic dependency on the dimension $d$.

The framework considered in this paper builds upon the previous OCO model by incorporating a sequence of time varying constraint functions $g^t_i(\mu),~i=1,2,\cdots,L$, which are also revealed at each time slot $t$ after the decision is made. The goal of this constrained OCO is to choose the decision sequence $\l\{\mu^0,\mu^1,~\cdots,~\mu^{T-1}\r\}\subseteq\Delta $ so that both the regret and constraint violations grow sublinearly in $T$ (i.e. 
$\sum_{t=0}^{T-1}g_i^t(\mu_t)\leq o(T)$) with respect to the best fixed decision in hindsight solving the following convex program:
\begin{equation}\label{eq:0}
\min_{\mu\in\Delta }\sum_{t=0}^{T-1}f^t(\mu),~~s.t.~~\sum_{t=0}^{T-1}g_i^t(\mu)\leq 0,~i=1,2,\cdots,L.
\end{equation}
The constrained OCO was first considered in the work \citep{Mannor09JMLR} where the authors (somewhat surprisingly) show via a counterexample that even with only one constraint, it is not always possible to achieve the aforementioned goal if we allow both objective and constraint functions to vary arbitrarily. Such an impossibility result implies that if one wants to obtain meaningful results on constrained OCO, then more assumptions have to be posed.

The works \citep{Mahdavi12JMLR, Jenatton16ICML, titov2018mirror} consider the scenario where the constraint functions are fixed (i.e. do not depend on the time index $t$) and propose primal-dual type methods whose analyses give $\mathcal{O}(T^{\max\{\beta,1-\beta\}})$ regret and $\mathcal{O}(T^{1-\beta/2})$ constraint violation, where $\beta\in[0,1]$ is an algorithm parameter. This bound is improved in the work \citep{yu2016low} where the authors show an $\mathcal{O}(\sqrt{T})$ regret bound and finite constraint violations (i.e. $\mathcal{O}(1)$ constraint violation) via Slater condition (i.e. There exists a $\mu\in\Delta $ such that $g_i(\mu)<0,~\forall i$). A more recent work \citep{yuan2018online} shows that one can get logarithm regret and $\mathcal{O}(\sqrt{T})$ constraint violations if one assumes instead that all objective functions are strongly convex.

Constrained OCO with stochastic constraints, where $g_i^t(\mu) = g_i(\mu,\gamma^t)$ and 
$\{\gamma^t\}_{t=0}^{T-1}$ are i.i.d., is considered in the works such as \citep{yu2017online,chen2019bandit,pmlr-v97-liakopoulos19a}, where a primal-dual proximal gradient algorithm is proposed and $\mathcal{O}(\sqrt{T})$ expected regret and constraint violations are shown under the Slater condition (i.e. there exists a $\mu\in\Delta $ such that $\expect{g_i(\mu,\omega^t)}<0,~\forall i$). Without Slater condition, the best known result is again $\mathcal{O}(T^{\max\{\beta,1-\beta\}})$ regret and $\mathcal{O}(T^{1-\beta/2})$ constraint violation as is shown in \citep{yi2019distributed}. Also, to the best of our knowledge, previous bounds in constrained online learning fail to recover the ``almost dimension free'' phenomenon for the probability simplex decision set ubiquitous in unconstrained scenarios.
In this paper, we make steps towards \textit{removing the Slater condition while maintaining the worst case optimal $\mathcal{O}(\sqrt{T})$ regret, constraint violations, and sharpening the dimension dependency on decision variables.}

Slater condition is assumed in the classical analysis of optimization algorithms for constrained convex programs such as the dual subgradient algorithm \citep{nedic2009approximate} and the interior point method \citep{boyd2004convex}. A key implication of Slater condition, which is adopted in the $\mathcal O(1/\sqrt{T})$ convergence rate analysis in \citep{nedic2009approximate}, is that it implies the existence and boundedness of Lagrange multipliers. However, the reverse implication is in general untrue, as one can show that for many equality constrained convex programs,  Lagrange multipliers do exist and are bounded \citep{bertsekas1999nonlinear}. This makes ``Slater condition free'' analysis an important topic in optimization theory and motivates series of improved primal-dual type algorithms and analysis for constrained convex programs with competitive convergence rate under the existence of Lagrange multipliers assumption \citep{neely2014simple,yurtsever2015universal,deng2017parallel,yu2017simple}. 

Replacing the Slater condition with Lagrangian type assumptions in online problems is highly non-trivial and does not follow from that of constrained convex programs. A key issue is that the objective function varies arbitrarily per slot, and so the definition of Lagrange multiplier is not clear. A simple attempt is to look at in-hindsight problems such as \eqref{eq:0} and see if the Lagrange multiplier of this problem helps with the regret analysis. However, since problem \eqref{eq:0} sums the objectives across the horizon, it hardly gives any insight on the per slot dynamics for any practical algorithm considered. If we instead look at the per slot constrained problem, then, one might be able to conduct analysis and obtain per-slot multipliers, but it is not clear how to piece together the analysis for different slots.

\subsection{Contributions}
 In this paper, we consider the stochastic constrained online learning problem and propose a new primal-dual online mirror descent framework, which simultaneously weakens the assumptions and improves the dimension factors in the previously known online proximal gradient type algorithms. We introduce a new \textit{sequential existence of Lagrange multipliers} condition, which is shown to be  \textit{strictly weaker} than the Slater condition, allows for equality constraints and bridges the aforementioned dilemma between on-hindsight problem and per slot problem. 
We then show via a new analysis that under such an assumption, the proposed algorithm enjoys a matching $\mathcal{O}(\sqrt{T})$ expected regret and constraint violations.  
For the  case when decisions are contained in a probability simplex, we reduce the dimension dependency 
to have only a logarithmic factor.
Conceptually, our analysis seems to be distinctive from the previous known methods in the sense that we look at the cumulative objectives over a specifically chosen time period (of length $\sqrt T$), and consider the following static constrained program starting from any time slot $t$:
$
\min_{\mu\in\Delta }\sum_{\tau=t}^{t+\sqrt T}\expect{f^{\tau}(\mu)},~~~s.t.~~ \expect{g_i(\mu,\omega^t)}\leq 0,
~~i=1,2,\cdots,L.
$
We demonstrate that the existence and boundedness of Lagrange multipliers for this problem provides certain weak error bound conditions for the dual function sufficient to bound the size of the dual variable process, leading to the desired results. 


\subsection{Notation} 
For any vector $\mathbf{v}\in\mathbb{R}^d$, $\mathbf{v}\geq0,~\mathbf{v}=0,~\mathbf{v}\leq0$ means $\mathbf{v}$ is entrywise nonnegative, zero and nonpositive, respectively. The notation $[\mf v]_+$ denotes entrywise application of the function $\max(x,0)$. The notation $\mathbb{R}^d_+$ stands for the positive orthant of  $\mathbb{R}^d$. 
For any set $\mathcal S\subseteq\mathbb{R}^d$, let $\text{int}(\mathcal{S})$ be its interior.
The norms $\|\mathbf{v}\|_1 := \sum_{i=1}^d|v_i|$, $\|\mathbf{v}\|_2 := (\sum_{i=1}^d|v_i|^2)^{1/2}$ and $\|\mathbf{v}\|_\infty := \max_{i}|v_i|$. For any convex function $f:\mathbb{R}^d\rightarrow\mathbb{R}$, we use $\nabla f(\mathbf{v})$ to denote any one of the subgradients at $\mathbf{v}$ and use $\partial f(\mathbf{v})$ to denote the set of all subgradients at $\mathbf{v}$. For any function $g(\mathbf{v},\xi)$ which is convex on the first argument $\mathbf{v}$, $\nabla g(\mathbf{v},\xi)$ denotes the subgradient of $g$ on $\mathbf{v}$ while fixing $\xi$.  For any closed set $K\subseteq\mathbb{R}^d$ and any point $\mf x\in\mathbb{R}^d$, the distance of  $\mf x$ to $K$ is defined as 
$\text{dist}(\mf x, K):=\min_{\mf y\in K}\|\mf{x-y}\|_2$.

\section{Problem Formulation and Algorithms}

\subsection{Basic definitions}\label{sec:def}
Let $\|\cdot\|$ be a general norm in $\mathbb{R}^d$. Define the dual norm on any $x\in\mathbb{R}^d$ as $\|x\|_*:= \sup_{\|y\|\leq1}\dotp{x}{y}$.
Consider a convex set $\mathcal C\subseteq\mb R^d$ (potentially be $\mb R^d$ itself) with a non-empty interior, i.e.  $\text{int}(\mathcal C)\neq\emptyset$. Let $\omega:\mathcal{C}\rightarrow\mathbb{R}$ be a function that is continuously differentiable in the interior of 
$\mathcal C$. Let $\Delta\subseteq \mathcal C$ be a \textit{compact convex} subset containing the origin and 
$\Delta^o:=\Delta\cap\text{int}(\mathcal C)$, which is non-empty. Define the \textit{Bregman divergence} function $D:\Delta\times \Delta^o\rightarrow\mathbb{R}$ generated from $\omega(\cdot)$ as follows:
\[
D(x,y):= \omega(x) - \omega(y) - \dotp{\nabla\omega(y)}{x-y}.
\]
The following is a key property of the Bregman divergence:
\begin{lemma}[Pushback]\label{lem:strong-convex}
Let $f:\mathcal{C}\rightarrow\mb{R}$ be a convex function. Fix $\alpha>0$, $y\in\Delta^o$. Suppose $x^*\in \text{argmin}_{x\in\Delta} f(x) + \alpha D(x, y)$ and $x^*\in\Delta^o$, then, for any $z\in\Delta$,
\[
f(x^*) + \alpha D(x^*,y)\leq f(z) + \alpha D(z,y) - \alpha D(z,x^*). 
\]
\end{lemma}
\begin{remark}
For the case where $f$ is a linear function and $\omega$ is convex, such a pushback result can be found, for example, in \citep{nemirovski2009robust}. For results with $f$ being on domain $\mathbb{R}^d$, the proof can be found in \citep{tseng2005}. Our result generalizes previous results to arbitrary set $\Delta$. It is proved in the Supplement (Section \ref{sec:property-divergence})
\end{remark}

We say $\omega(\cdot)$ is a \textit{distance generating function} if for any 
$x\in\text{int}(\mathcal C)$, $\omega(\cdot)$
is a continuously differentiable and strongly convex with modulus $\beta$ with respect to the primal norm $\|\cdot\|$, i.e.
$
\dotp{x - y}{\nabla\omega(x) - \nabla\omega(y)}\geq \beta \|x-y\|^2,~\forall x,y\in \text{int}(\mathcal C).
$
It is easy to see if $\omega$ is a distance generating function, then, the corresponding $D(\cdot,\cdot)$ satisfies 
\begin{equation}\label{eq:strong-convexity}
D(x,y)\geq \l.\beta\|x-y\|^2\r/2,~~\forall x,y\in\text{int}(\mathcal{C}).
\end{equation}
Note that $D(x,y)$ behaves asymmetrically on $x$ and $y$ over potentially different domains, which results from the (possible) non-differentiability of the distance generating function $\omega(\cdot)$ on the boundary of $\Delta$. One such example is the KL divergence. 

\begin{enumerate}[leftmargin=*]
\item    The set 
$\Delta =  \{\mu\in\mathbb{R}^d:~\|\mu\|_1 = 1,~\mu\geq0\}$ is a probability simplex, $\mathcal C = \mathbb{R}^d_+$,
the function
$\omega(\mu)=-\sum_{i=1}^d\mu_i\log\mu_i$ is the entropy function, and for any two distributions $\mu^a\in\Delta,~\mu^b\in\Delta^o$, $D(\mu^a,\mu^b):=\sum_{i=1}^d\mu^a_i\log(\mu^a_i/\mu^b_i)$ is the well-known Kullback-Leibler (KL) divergence. Furthermore, by Pinsker's inequality, it is strongly convex with respect to $\|\cdot\|_1$ with the strongly convex modulus $\beta=1$.
The dual norm in this space is $\|\cdot\|_\infty$.
\item  The set $\Delta$ is in the Euclidean space $\mathbb R^d$, $\mathcal{C} =\mathbb R^d$ and $\omega(x) = \frac12\|x\|_2^2$, which is strongly convex with respect to $\|\cdot\|_2$, $D(x,y) = \|x-y\|_2^2$, and the dual norm is also $\|\cdot\|_2$.
\end{enumerate}

\vspace{-1em}
\subsection{Problem formulation}\label{sec:formulation}
\vspace{-0.5em}
In this section, we set up the basic formulation of stochastic constrained online optimization. 
Let $\{\xi^t\}_{t=0}^\infty$ and
$\{\gamma^t\}_{t=0}^{\infty}$ be two processes, where $\{\xi^t\}_{t=0}^\infty$ can be arbitrarily time varying (might be chosen based on the system history) 
and $\{\gamma^t\}_{t=0}^{\infty}$ are i.i.d. realizations of a random variable $\gamma$ with a possibly unknown distribution.
Let $f(\mu,\xi^t)$, $g_i(\mu, \gamma^t), i\in \{1,2,\ldots, L\}$ be deterministic functions which are convex in the first component given the second component. Furthermore, let $\{h_j^t\}_{t=0}^{\infty},~j\in\{1,2,\cdots,M\}$ be sequences of i.i.d. random vectors in $\mathbb{R}^d$. Throughout the paper, we assume $\xi^t,\gamma^t,h_j^t$ are jointly independent for all $t$ with system history up to time $t$ as 
$\mathcal{F}_t := \{\xi^\tau,\gamma^\tau,h_j^\tau\}_{\tau=0}^{t-1}$.
For any fixed $\mu\in\Delta$, we write $f^t(\mu) := f(\mu,\xi^t)$, $g_i^t(\mu) := g_i(\mu, \gamma^t)$, and $\overline f^t(\mu) = \mathbb{E}[f^t(\mu)|\mathcal{F}_t]$, $\overline g_i(\mu) = \mathbb{E}[g_i^t(\mu)]$.  We further define the vectorized notations $\mathbf{g}^t(\mu)= [g^{t}_{1}(\mu), \ldots, g^{t}_{L}(\mu)]\tran$, $\overline{\mathbf{g}}(\mu) = [\mathbb{E}[g_{1}(\mu, \omega)], \ldots, \mathbb{E}[g_{L}(\mu, \omega)]]\tran$, $\mathbf{h}^t(\mu)= [\dotp{h^{t}_{1}}{\mu}, \ldots, \dotp{h^{t}_{M}}{\mu}]\tran$ and 
 $\overline{\mathbf{h}}(\mu)= [\dotp{\expect{h^{t}_{1}}}{\mu}, \ldots, \dotp{\expect{h^{t}_{M}}}{\mu}]\tran$. It is also worth noting that our algorithms and analysis also apply to the special case where $\{\xi^t\}_{t=0}^\infty$ are also i.i.d. for which we have 
$\overline f^t(\mu) = \mathbb{E}[f^t(\mu)]$.

Define the benchmarking decision in-hindsight $\mu^*$ as a solution to the following static convex program:
\vspace{-0.5em}
\begin{equation}\label{eq:on_hindsight}
\min_{\mu\in\Delta}~~ \sum_{t=0}^{T-1}\overline f^t(\mu)~~s.t.~~\overline{\mathbf{g}}(\mu)\leq0,~~\overline{\mathbf{h}}(\mu) =\mf b,
\end{equation}
where $\mf b = [b_1,~b_2,~\cdots,~b_M]\tran$ is a vector of constants.
At the beginning of each time slot $t$, none of the objective function $f^{t}(\mu)$, constraint function 
$g_{i}^t(\mu)$ or random vector $h_j^t$ is known. The decision maker is supposed to choose a vector $\mu^t\in\Delta$ first before observing these quantities. The goal is to make sequential (possibly randomized) decisions so that both the expected regret, defined as 
$\sum_{t=0}^{T-1}\expect{f^t(\mu^t) - f^t(\mu^*)}$, 
and expected constraint violations, define as $\sum_{t=0}^{T-1}\expect{g_i^t(\mu^t)}$ and 
$\mathbb E|\sum_{t=0}^{T-1}h_j^t(\mu^t)|$, grow sublinearly with respect to the time horzon $T$.
Throughout this paper, we make the following boundedness assumption:

\begin{Assumption}[Boundedness of objectives and constraint functions] \label{as:basic}~
\begin{enumerate}[leftmargin=*]
\item Objective functions $f^t(\mu)$ and constraint functions $g_i^t(\mu)$ have bounded subgradients on $\Delta$, i.e. there exist absolute constants $D_{1}>0 $  and $D_{2}>0$ such that $\| \nabla f^{t} (\mu)\|_* \leq D_{1}$, $\sum_{i=1}^L\| \nabla g_{i}^t(\mu) \|_*^2 \leq D_{2}^2$, for all $\mu\in \Delta$, all $t\in \{0,1,\ldots\}$, and all $ i\in\{1,2,\ldots,L\}$.
\item There exist absolute constants $F,G,H>0$ such that $|f^t(\mu)|\leq F,~\forall t\in\{0,1,2,\cdots\}$, $\sum_{i=1}^L|g_i^t(\mu)|^2 \leq G^2$ for all $\mu\in \Delta,~ t\in\{0,1,2,\cdots\}$, 
and $\sum_{j=1}^M\|h_j^t\|_*^2\leq H^2$, for all $j\in\{1,2,\cdots,M\},~t\in \{0,1,\ldots\}$.
\item The Bregman divergence $D(\cdot,\cdot)$ is generated from a distance generating $\omega(\cdot)$ and bounded on the set $\Delta$, i.e. there exists a constant $R$ such that 
$\sup_{x\in\Delta,y\in\Delta^o}D(x,y)\leq R$.
\end{enumerate}
\end{Assumption}

By strong convexity of the Bregman divergence \eqref{eq:strong-convexity}, we have $\sup_{x\in\Delta,y\in\Delta^o}\|x- y\|^2\leq 2R/\beta$. 
Note further that KL divergence does not satisfy Assumption \ref{as:basic}(3), for which we will develop a separate new algorithm in Section \ref{sec:prob-simplex}.

\subsection{Primal-dual online mirror descent}
We are now in a position to introduce our new online mirror descent (Algorithm \ref{alg:new-alg}) for the stochastic constrained online learning. The algorithm computes the next decision $\mu^{t+1}$ by a proximal mirror map using $\mu^t$, $f^t$ and $g^t_i$, and control the constraint violations via dual multipliers $\mf Q(t)$ and $\mf H(t)$.
 \begin{algorithm}
\caption{}
\label{alg:new-alg}
Let $V, \alpha> 0$ be some trade-off parameters. Let $Q_i(t),~H_j(t)$ be sequences of dual multipliers such that $Q_i(0) = 0,~H_j(0) = 0,~\forall i,j$. Let $\mu^{0}=\mu^{-1} \in\Delta$. 

\textbf{For} t = 0 to $T-1$:
\begin{enumerate}[leftmargin=15pt]
\item  Choose $\mu^{t}$ as a solution to the following problem:
\vspace{-0.5em}
\begin{align}
 \min_{\mu\in \Delta} \Big\langle V\nabla f^{t-1}(\mu^{t-1}) + \sum_{i=1}^{L} Q_i(t)\nabla g_i^{t-1}(\mu^{t-1}) 
 + \sum_{j=1}^{M} H_i(t)h_i^{t-1}, \mu\Big\rangle + \alpha D(\mu,\mu^{t-1})   \label{eq:mu-update}
\end{align}
\item Update each dual multiplier $Q_i(t), H_j(t)$ via 
\begin{align}
&Q_{i}(t+1) = \max\left\{Q_{i}(t) + g^{t-1}_{i}(\mu^{t-1}) + \dotp{\nabla g_{i}^{t-1} (\mu^{t-1})}{\mu^t-\mu^{t-1}}, 0\right\},  ~i\in\{1,2,\cdots,L\}\label{eq:Q-update}\\
&H_{j}(t+1) = H_{j}(t) + \dotp{h_{j}^{t-1}}{\mu^t} - b_j,
~j\in\{1,2,\cdots,M\}\label{eq:H-update}
\end{align}
\item Observe the objective function $f^{t}$ and constraint functions $\{g_i^t\}_{i=1}^L,~\{h_j^t\}_{j=1}^M$.
\end{enumerate}
\textbf{End for}.
\end{algorithm}

\vspace{-0.5em}
\subsection{Sequential Existence of Lagrange Multipliers (SELM)}
In this section, we introduce our Lagrange multiplier condition. A detailed comparison between such a condition and other constraint qualification conditions is delayed to the Supplementary (Section \ref{sec:constraint-qualification}). We start by defining a partial average function starting from any time slot $t$ as: $\overline{f}^{t,k}: = \frac1k\sum_{i=0}^{k-1}\overline{f}^{t+i}$. Consider the following optimization problem:
\begin{equation}\label{eq:partial-static-prob}
\min_{\mu\in\Delta}\overline{f}^{t,k}(\mu)~~s.t.~~\overline{\mathbf{g}}(\mu)\leq0,~~\overline{\mathbf{h}}(\mu) = \mf b,
\end{equation}
where $\overline{\mathbf{g}}(\mu), ~\overline{\mathbf{h}}(\mu)$ are defined in Section \ref{sec:formulation}.
Denote the solution to this program as $\overline{f}^{t,k}_*$. 
Define the Lagrangian dual function of \eqref{eq:partial-static-prob} as 
\begin{equation}\label{eq:target-dual}
q^{(t,k)}(\lambda,\eta):= \min_{\mu\in\Delta} \overline{f}^{t,k}(\mu) + \sum_{i=1}^L\lambda_i\overline{g}_i(\mu) 
+ \sum_{j=1}^M\eta_j(\overline{h}_j(\mu)-b_j),
\end{equation}
where $\lambda\in\mathbb{R}^L_+$ and $\eta\in\mathbb{R}^M$ are dual variables. For simplicity of notations, we always enforce them to be row vectors. Now, we are ready to state our condition:
\begin{Assumption}[Sequential existence of Lagrange multipliers (SELM)]\label{as:selm}
For any time slot $t$ and any time period $k$, the set of primal optimal solution to \eqref{eq:partial-static-prob} is non-empty. Furthermore the set of dual optimal solution, which is 
$\mathcal{V}_{t,k}^* := \text{argmax}_{\lambda\in\mathbb{R}^L_+,~\eta\in\mathbb{R}^M}q^{(t,k)}(\lambda,\eta)$, is non-empty and bounded. Any vector in $\mathcal{V}^*$ is called a Lagrange multiplier associated with \eqref{eq:partial-static-prob}.
Furthermore, there exists an absolute constant $B>0$ such that for any $t\in\{0,1,\cdots,T-1\}$ and $k=\sqrt{T}$, the dual optimal set 
$\mathcal{V}_{t,k}^*$ defined above satisfies $\max_{[\lambda,\mu]\in\mathcal{V}_{t,k}^*}\|[\lambda,\mu]\|_2\leq B$.
\end{Assumption}

\begin{remark}
Note first that SELM reduces to the known existence and boundedness of Lagrange multipliers assumption adopted in optimization theory when the objectives are also i.i.d. functions. 
In Section \ref{sec:constraint-qualification} of the Supplement, we show that SELM is equivalent to certain constraint qualification conditions and strictly weaker than the Slater conditions. In particular, we obtain the following simplifications in special cases: (1) Lemma \ref{lem:slater-selm} shows that Slater condition implies SELM. (2) Corollary \ref{cor:mfcq-2} implies that
when the interior of $\Delta$ is non-empty and there are only equality constraints, the linear independence of $\{\expect{h^t_1},~\expect{h^t_2},~\cdots,~\expect{h^t_M}\}$ is equivalent to SELM. (3) Lemma \ref{cor:mfcq} implies that when $\Delta$ is a probability simplex there are only equality constraints, the linear independence of $\{\mathbf{1},~\expect{h^t_1},~\expect{h^t_2},~\cdots,~\expect{h^t_M}\}$ is equivalent to SELM.
\end{remark}

The motivation for SELM is as follows: whenever Lagrange multipliers exist and are bounded, we automatically get that the dual function deviates according to a certain curve related to the distance from the set of Lagrange multipliers, namely, the weak error bound condition (EBC).
\begin{definition}[Weak error bound condition (EBC)]\label{def:ebc}
Let $F(\mathbf{x})$ be a concave function over $\mathbf{x}\in\mathcal{X}$, where $\mathcal{X}$ is closed and convex. 
Suppose $\Lambda^*:= \argmax_{\mathbf{x}\in\mathcal{X}}F(\mathbf{x})$ is non-empty. The function $F(\mathbf{x})$ satisfies the weak EBC if there exists constants $\ell_0,~c_0>0$ such that for any $\mf x\in\mathcal{X}$ satisfying 
$\text{dist}(\mf x,\Lambda^*)\geq\ell_0$,
\begin{equation*}
 F(\mf x^*)-F(\mf x) \geq c_{0}\cdot \text{dist}(\mf x,\Lambda^*).
\end{equation*}
\end{definition}
\vspace{-5pt}
Note that in Definition \ref{def:ebc}, $\Lambda^*$ is a closed convex set. This follows from the fact that $F(\mf x)$ is a convex function and thus all sub level sets are closed and convex. 
The following lemma shows SELM implies weak EBC:
\begin{lemma}\label{lem:leb}
Fix $T\geq 1$. Suppose Assumption \ref{as:selm} holds, then for any $t\in\{0,1,\cdots,T-1\}$ and $k=\sqrt{T}$,  there exists constants $c_0,\ell_0>0$,  such that the dual function $-q^{(t,k)}(\lambda,\eta)$ defined in \eqref{eq:target-dual} satisfies the weak EBC with parameter $c_0,\ell_0$. 
\end{lemma}
In the Supplement (Section \ref{sec:selm-ebc}), we will compare this weak EBC with the classical EBC in optimization theory and show that classical EBC implies weak EBC with explicit constants.

\vspace{-10pt}

\section{Main results}
\vspace{-5pt}
In this section, we present our main result of online primal-dual mirror descent. 
\begin{theorem}\label{main:theorem}
Let $\mu^*$ be a solution to the in-hindsight optimization problem \eqref{eq:on_hindsight}.
Suppose Assumption \ref{as:basic} and \ref{as:selm} hold. Let $\overline{c},~\overline{\ell}>0$ be absolute constants such that 
$c_0\geq\overline{c}$ and $\ell_0\leq\overline{\ell}$ for all $c_0,~\ell_0$ obtained in Lemma \ref{lem:leb} over $t=0,1,2,\cdots, T-1$ and $k=\sqrt{T}$.
If we choose $\alpha = T,V=\sqrt{T}$ in Algorithm \ref{alg:new-alg}, then the expected regret and constraint violations satisfy:
\begin{align*}
&\frac1T\sum_{t=0}^{T-1}\expect{f^t(\mu^t) - f^t(\mu^*)}\leq \frac{C_0^{\prime}}{\sqrt{T}},\\
&\mathbb{E}\Big\Vert\Big[\frac1T\sum_{t=0}^{T-1}\overline{\mathbf{g}}(\mu^t)\Big]_+\Big\Vert_2
\leq  \frac{C_1^{\prime}}{\sqrt{T}},~~~  
\mathbb{E}\Big\Vert\frac1T\sum_{t=0}^{T-1}\overline{\mf h}(\mu^t) - \mf b\Big\Vert_2\leq  \frac{C_2^{\prime}}{\sqrt{T}},  
\end{align*}
\vspace{-0.1em}
where $C_0^{\prime}, C_1^{\prime}, C_{2}^{\prime}$ are constants depending linearly on $D_1^2+D_1+ D_2^2+G^2+H^2+G+H+F$ and independent of $T$.
\end{theorem}

\vspace{-1em}
\subsection{Proof of regret bound}\label{sec:reg-analysis}
\vspace{-0.5em}
In this section, we present the proof of regret bound in Theorem \ref{main:theorem}. The proofs of technical lemmas are delayed to the Supplement (Section \ref{sec:pf-regret}).
We start with the following key bound of a drift-plus-penalty (DPP) expression:
\begin{lemma}\label{lem:strong-convex-queue}
Define the drift $\Delta(t):= (\|\mf Q(t+1)\|_2^2-\|\mf Q(t)\|_2^2)/2+(\|\mf H(t+1)\|_2^2-\|\mf H(t)\|_2^2)/2$.
Consider the following ``drift-plus-penalty'' (DPP) expression at time $t$:
$V\dotp{\nabla f^{t-1}(\mu^{t-1})}{\mu^t-\mu^{t-1}}+\Delta(t) + \alpha D(\mu^t, \mu^{t-1})$. 
Let $M= \frac{4RH^2}{\beta} + G^2+\frac{2RD_2^2}{\beta}$ where $\beta$ is in \eqref{eq:strong-convexity}, then, for any $\mu\in\Delta$,
\begin{multline}
V\dotp{\nabla f^{t-1}(\mu^{t-1})}{\mu^t-\mu^{t-1}}+\Delta(t) + \alpha D(\mu^t, \mu^{t-1})
 \leq  V(f^{t-1}(\mu) - f^{t-1}(\mu^{t-1})) \\
 + \sum_{i=1}^LQ_i(t)g^{t-1}_{i}(\mu) + \sum_{j=1}^{M} H_j(t)(\dotp{h_j^{t-1}}{\mu}-b_j)
 + \alpha D(\mu,\mu^{t-1}) - \alpha D(\mu, \mu^t) + M
 .  \label{eq:interim-1}
\end{multline}
\end{lemma}

This lemma is proved via the property of Bregman divergence (Lemma \ref{lem:strong-convex}). Now, for the DPP expression on the left hand side, we also have the following lower bound:
\begin{lemma}\label{lem:lhs-lower-bound}
Our Algorithm \ref{alg:new-alg} ensures
\begin{equation}\label{eq:res-bound}
V\dotp{\nabla f^{t-1}(\mu^{t-1})}{\mu^t-\mu^{t-1}} + \alpha D(\mu^t, \mu^{t-1})
\geq - V^{2}D_1^2/2\alpha\beta.
\end{equation}
\end{lemma}
Substituting this bound in to \eqref{eq:interim-1}, taking $\mu=\mu^*$ which is the solution to the in-hindsight problem \eqref{eq:on_hindsight}, and taking conditional expectations from both sides, we readily get:
{\small
\begin{multline}\label{eq:roadmap-2}
-\frac{V^{2}}{2\alpha\beta}D_1^2 +  \expect{\Delta(t)|\mathcal{F}_{t-1}}\leq V\expect{f^{t-1}(\mu^*) - f^{t-1}(\mu^{t-1}) |\mathcal{F}_{t-1}} 
+
\mathbb{E}\Big[\sum_{i=1}^LQ_i(t)g^{t-1}_{i}(\mu^*)\Big|\mathcal{F}_{t-1}\Big] \\
+ \mathbb{E}\Big[\sum_{j=1}^{M} H_j(t)(\dotp{h_j^{t-1}}{\mu^*}-b_j)~\Big|\mathcal{F}_{t-1}\Big]
+ \alpha \expect{D(\mu^*,\mu^{t-1}) - D(\mu^*, \mu^t)~|\mathcal{F}_{t-1}}  + M  .
\end{multline}
}%
\vspace{-0.1em}
Note that 
{\small
$$ \mathbb{E}\Big[\sum_{j=1}^{M} H_j(t)(\dotp{h_j^{t-1}}{\mu^*}-b_j)\Big|\mathcal{F}_{t-1}\Big] = 
\sum_{j=1}^{M} H_j(t)\expect{\dotp{h_j^{t-1}}{\mu^*} - b_j} = 0,$$
$$\mathbb{E}\Big[\sum_{i=1}^LQ_i(t)g^{t-1}_{i}(\mu^*)~\Big|\mathcal{F}_{t-1}\Big] = \sum_{i=1}^LQ_i(t)\expect{g^{t-1}_{i}(\mu^*)}\leq0,$$ 
}%
where, in both inequalities, the first step follows from the fact that $h^t_j, g_i^t$ are i.i.d. and $H_j(t),Q_i(t)$ depend on $\mathcal{F}_{t-1}$, and the second step follows from 
$\mu^*$ being a solution to the in-hindsight optimization problem \eqref{eq:on_hindsight}, thus, must be feasible, i.e. 
$\expect{g^{t-1}_{i}(\mu^*)} \leq 0$, $\expect{\dotp{h_j^{t-1}}{\mu^*}}=0$.
Thus, taking the full expectation from both sides of \eqref{eq:roadmap-2} gives
\vspace{-0.5em}
\begin{multline*}
\expect{\Delta(t)} + V\expect{f^{t-1}(\mu^{t-1}) - f^{t-1}(\mu^*)} 
\leq M
+   \frac{V^{2}D_1^2}{2\alpha\beta}
+ \alpha \expect{D(\mu^*,\mu^{t-1}) - D(\mu^*, \mu^t)} .
\end{multline*}
Taking a telescoping sum on both sides from 0 to $T-1$ and dividing both sides by $TV$,
\begin{align*}
\frac1T\sum_{t=0}^{T-1}\expect{f^{t-1}(\mu^{t-1}) - f^{t-1}(\mu^*)} 
\leq 
\frac{M}{V}
+   \frac{VD_1^2}{2\alpha\beta}
+ \frac{\alpha}{VT} D(\mu^*,\mu^0),
\end{align*}
where we use the fact that since $Q_i(0)=0$ and $H_j(0) = 0$,
$\sum_{t=0}^{T-1}\Delta(t) = (\|\mf Q(T)\|_2^2 + \|\mf H(T)\|_2^2)/2\geq0$. Substituting $\alpha = T, V = \sqrt{T}$, and $D(\mu^*,\mu^0)\leq R$ yields the desired result with $C_0' = \frac{RH^2}{\beta } + G^{2} +\frac{2RD_2^2}{\beta} + \frac{D_1^2}{2\beta} +  R$.

\subsection{Proof of constraint violations}\label{sec:constraint-violation}
In this section, we present the proof of constraint violations in Theorem \ref{main:theorem}. The proofs of technical lemmas are delayed to the Supplement (Section \ref{sec:pf-constraint-1}-\ref{sec:pf-constraint-2}).
First, it is enough to bound dual multipliers via the following lemma:
\begin{lemma}\label{lem:constraint}
The updating rule \eqref{eq:Q-update} and \eqref{eq:H-update} delivers the following constraint violation bounds:
{
\begin{align*}
&\mathbb{E}\Big\Vert\Big[\frac1T\sum_{t=0}^{T-1}\overline{\mathbf{g}}(\mu^t)\Big]_+\Big\Vert_2
\leq \frac{\expect{\|\mf Q(t)\|_2}}{T} + \frac{VD_{1}D_2}{\alpha\beta} + \frac{1}{T}\sum_{t=1}^{T} \frac{D_{2}}{\alpha \beta}\l(D_2\expect{\|\mf Q(t)\|_2} + H\expect{\|\mf H(t)\|_2}\r) \\
&\mathbb{E}\Big\Vert\frac1T\sum_{t=0}^{T-1}\overline{\mf h}(\mu^t)-\mf b\Big\Vert_2
\leq \frac{\expect{\|\mf H(t)\|_2}}{T} + \frac{VD_{1}H}{\alpha\beta}+\frac{1}{T}\sum_{t=1}^{T}\frac{H}{\alpha\beta}\l(D_2\expect{\|\mf Q(t)\|_2} + H\expect{\|\mf H(t)\|_2}\r) 
\end{align*}
}%
\end{lemma}

To bound $\expect{\|\mf Q(t)\|_2}$ and $\expect{\|\mf H(t)\|_2}$, we have the following lemma:
\begin{lemma}\label{lem:dual-bound}
Define constant
{$C_{V,\alpha,t_0} := 2\big(\frac{4RH^2}{\beta} + G^{2}+\frac{2RD_2^2}{\beta}+  \frac{V^{2}}{2\alpha\beta}D_1^2 + VF \big)t_0 + 2\big(\frac32G^2 + \frac{2RD_2^2}{\beta} + \frac{8RH^2}{\beta}\big)t_0^{2}+ 2\alpha R.$} Then, for any integer $t_{0}\geq 1$, we have the $t_0$ step drift satisfies
\begin{align}\label{eq:drift-bound-1}
&\expect{\|\mf Q(t+t_0)\|_2^2 + \|\mf H(t+t_0)\|_2^2~|\mathcal F^{t-1}} - \|\mf Q(t)\|_2^2 - \|\mf H(t)\|_2^2 \nonumber \\
\leq& 2Vt_0\expect{\l.q^{(t-1,t_0)}\big(\frac{\mf Q(t)}{V},~\frac{\mf H(t)}{V}\big) ~\r|\mathcal{F}^{t-1}} + C_{V,\alpha,t_0}.
\end{align}
where the dual function $q^{(t-1,t_0)}$ is defined in \eqref{eq:target-dual}.
\end{lemma}
This bound establishes the relation between dual multipliers and the dual function. 
Next, in view of \eqref{eq:drift-bound-1}, we would like to show that $\expect{\l.q^{(t-1,t_0)}\big(\frac{\mf Q(t)}{V},~\frac{\mf H(t)}{V}\big) ~\r|\mathcal{F}^{t-1}}$ is small. This is done via Lemma \ref{lem:leb} that whenever 
$\big(\frac{\mf Q(t)}{V},~\frac{\mf H(t)}{V}\big)$ is far away from the optimal set 
$\mathcal{V}_{t-1,t_0}^*:=\text{argmax}_{\lambda,\eta}q^{(t-1,t_0)}\big(\lambda,~\eta\big)$, which is nonempty and bounded by Assumption \ref{as:selm}, $\expect{\l.q^{(t-1,t_0)}\big(\frac{\mf Q(t)}{V},~\frac{\mf H(t)}{V}\big) ~\r|\mathcal{F}^{t-1}}$ becomes negative. In fact one can prove the following lemma:
\begin{lemma}\label{lem:bound-dual-function}
{
The dual function has the following bound:
\[
\expect{q^{(t-1,t_0)}\big(\frac{\mf Q(t)}{V},~\frac{\mf H(t)}{V}\big) ~|\mathcal{F}^{t-1}}
\leq F + \overline{\ell}(G+\sqrt{2RH^2/\beta} +\overline{c}) + \overline{c}B -\overline{c}\Big\|\big(\frac{\mf Q(t)}{V},~\frac{\mf H(t)}{V}\big)\Big\|_2,
\]
where} $B$ is defined in Assumption \ref{as:selm}.
\end{lemma}
Substituting the above lemma into \eqref{eq:drift-bound-1} and using a known stochastic drift lemma, one can prove the following bound by setting $t_0=\sqrt{T}$, $V=\sqrt{T},~\alpha = T$: 
\begin{lemma}\label{lem:q-bound}
The quantity  $\| \big(\mf Q(t),~\mf H(t)\big) \|_2$ satisfies the following conditions:
\begin{equation}\label{eq:Q-bounds}
\expect{\Big\|\big(\mf Q(t),~\mf H(t)\big)\Big\|_2} \leq C^{\prime} + C^{\prime\prime} \sqrt{T}
\end{equation}
where  $C^{\prime} := \frac{2}{\overline{c}}\big(\frac{4RH^2}{\beta} + G^{2}+\frac{2RD_2^2}{\beta}+  \frac{D_1^2}{2\beta}\big)$ and $C^{\prime\prime}:= \frac{2}{\overline{c}}\big(2F + \frac{3}{2}G^2 + \frac{2RD_2^2}{\beta} + \frac{8RH^2}{\beta} + R + \overline{\ell}(G+\sqrt{8RH^2/\beta} +\overline{c}) + \overline{c}B + 2\big(2 (G+\sqrt{2RD_2^2/\beta}) + \sqrt{8RH^2/\beta}\big)^2 \log\big(\frac{8 (2(G+\sqrt{\frac{2RD_2^2}{\beta}}) + \sqrt{\frac{8RH^2}{\beta}})^2}{\overline c^2}\big)\big)$ are absolute constants.
\end{lemma}
Substituting the bound \eqref{eq:Q-bounds} into Lemma \ref{lem:constraint} with $\alpha= T$ and $V=\sqrt{T}$  gives the final constraint violation bounds.

\vspace{-0.7em}
\section{The probability simplex case}\label{sec:prob-simplex}
\vspace{-5pt}
{ In this section, we deal with the probability simplex case where the decision set $\Delta$ is a $d$-dimensional probability simplex with huge $d$.  While Algorithm \ref{alg:new-alg} can be applied to solve such problems by choosing $D(\mu, \mu^{t-1})$ to be $\|\mu-\mu^{t-1}\|_2^{2}$, due to the dependencies on the $D_1, D_2,G,H,F$,
 the constant factors in Theorem \ref{main:theorem} linearly depend on $d$. For mirror descent over a probability simplex, to improve the dimension dependence,  people usually choose the Bregman divergence distance $D(\cdot, \cdot)$ to be the KL divergence. 
However, KL divergence fundamentally violates the third assumption in Assumption \ref{as:basic}.  We now present an alternative algorithm in Algorithm \ref{alg:prob-simplex} and shows that it can achieve sublinear regret and constraint violations that logarithmically depends on $d$ .}
 \vspace{-5pt}

\begin{algorithm} 
\caption{}
\label{alg:prob-simplex}
Let $V, \alpha> 0$,~$\theta\in[0,1)$ be some trade-off parameters. Let $D(\mu_{1}, \mu_{2}) = \sum_{i=1}^{d} \mu_{1}(i) \log \frac{\mu_{1}(i)}{\mu_{2}(i)}$. Let $Q_i(t),~H_j(t)$ be sequences of dual multipliers such that $Q_i(0) = 0,~H_j(0) = 0,~\forall i,j$. Let $\mu_0=\mu_{-1} =  \frac1d\mathbf{1}$. 

\textbf{For} t = 0 to $T-1$:
\begin{enumerate}[leftmargin=15pt]
\item Let $\tilde{\mu}^{t-1} = (1-\theta)\mu^{t-1} + \frac{\theta}{d}\mathbf{1}$.
\item  Choose $\mu^{t}$ as a solution to the following problem:
\vspace{-0.3em}
\begin{align}
 \min_{\mu\in \Delta} \Big\langle V\nabla f^{t-1}(\mu^{t-1}) + \sum_{i=1}^{L} Q_i(t)\nabla g_i^{t-1}(\mu^{t-1}) 
 + \sum_{j=1}^{M} H_i(t)h_i^{t-1}, \mu \Big\rangle + \alpha D(\mu,\tilde{\mu}^{t-1})   \label{eq:mu-update}
\end{align}
\item Update each dual multiplier $Q_i(t), H_j(t)$ via \eqref{eq:Q-update} and \eqref{eq:H-update}. 
\item Observe the objective function $f^{t}$ and constraint functions $\{g_i^t\}_{i=1}^L,~\{h_j^t\}_{j=1}^M$.
\end{enumerate}
\textbf{End for}.
\end{algorithm}

Compared to Algorithm \ref{alg:new-alg}, Algorithm \ref{alg:prob-simplex} uses the K-L divergence as the particular Bregman divergence and introduces a probability mixing step $\tilde{\mu}^{t-1} = (1-\theta)\mu^{t-1} + \frac{\theta}{d}\mathbf{1}$, which pushes the update away from the boundary, at each round. Furthermore, it is known that the problem \eqref{eq:mu-update} admits a closed form solution known as the exponential gradient update \citep{Hazan16FoundationTrends}. More specifically, define
\vspace{-0.6em} 
$$\mathbf{p}^{t-1} := \alpha^{-1}\big(V\nabla f^{t-1}(\mu^{t-1}) + \sum_{i=1}^{L} Q_i(t)\nabla g_i^{t-1}(\mu^{t-1}) 
 + \sum_{j=1}^{M} H_i(t)h_i^{t-1}\big).$$ 
Then, the update $\mu^t$ can simply be written as $\mu^t_i = \frac{\tilde{\mu}^{t-1}_i\exp(-p_i^{t-1})}{\sum_{k=1}^d\tilde{\mu}^{t-1}_k\exp(-p_k^{t-1})},~~i\in\{1,2,\cdots,d\}.$

We have the following performance bound on this algorithm whose proof is similar to Theorem \ref{main:theorem} and delayed to the Supplement (Section \ref{sec:pf-prob-simplex}):

\begin{theorem}\label{main:theorem-simplex}
Suppose the first two in Assumption \ref{as:basic} (using $\Vert\cdot\Vert=\Vert\cdot\Vert_{1}$ and $\Vert\cdot\Vert_{\ast}=\Vert\cdot\Vert_{\infty}$) and Assumption \ref{as:selm} hold. Let $\overline{c},~\overline{\ell}>0$ be absolute constants such that 
$c_0\geq\overline{c}$ and $\ell_0\leq\overline{\ell}$ for all $c_0,~\ell_0$ obtained in Lemma \ref{lem:leb} over $t=0,1,2,\cdots, T-1$ and $k=\sqrt{T}$. Choose $\alpha = T,~V=\sqrt{T}$, $\theta = 1/T$ in Algorithm \ref{alg:prob-simplex}.
The expected regret and constraint violations satisfy:
\vspace{-0.5em}
{\small
\begin{align*}
&\frac1T\sum_{t=0}^{T-1}\expect{\overline{f}^t(\mu^t) - f_t(\mu^*)}\leq \frac{\hat{C}_{0}^{\prime}}{\sqrt{T}} +\frac{\hat{C}_{0}^{\prime}\log(d)}{\sqrt{T}} \\
&\mathbb{E}\l\|\l[\frac1T\sum_{t=0}^{T-1}\overline{\mathbf{g}}(\mu^t)\r]_+\r\|_2
\leq  \frac{\hat{C}^{\prime}_{1}}{\sqrt{T}} + \frac{\hat{C}^{\prime\prime}_{1}\log(Td)}{\sqrt{T}},\\ 
&\mathbb{E}\l\|\frac1T\sum_{t=0}^{T-1}\overline{\mf h}(\mu^t) - \mf b\r\|_2\leq  \frac{\hat{C}^{\prime}_{2}}{\sqrt{T}} + \frac{\hat{C}^{\prime\prime}_{2}\log(Td)}{\sqrt{T}}. 
\end{align*}
}%
where $\hat{C}_0^{\prime}, \hat{C}_1^{\prime}, \hat{C}_{1}^{\prime\prime}, \hat{C}_{2}^{\prime}, \hat{C}_{2}^{\prime\prime}$ are absolute constants depending linearly on $D_1^2+D_1+ D_2^2+G^2+H^2+G+H+F$ and independent of $d$ or $T$. (Note that $D_{1}, D_{2}, G, H, F$ in Assumption \ref{as:basic} are independent of $d$ when $\Vert\cdot\Vert_{\ast}=\Vert\cdot\Vert_{\infty}$.)
\end{theorem}

\section{Simulation experiments}

We consider the problem of cost minimization under budget pacing constraints in data center service scheduling. More specifically, consider a geographically distributed data center consists of 5 server clusters serving one stream of incoming jobs arriving at a central controller. Each cluster contains 10 servers. The jobs are directed to different clusters for processing by controller with different per unit electricity costs. In the simulation, we use electricity market price (EMP) data traces from 5 zones of New York ISO open access pricing data (\url{http://www.nyiso.com/}). For example, Fig \ref{fig}(a) depicts the per 5 min EMP data of zone DUNWOD between 05/01/2017 and 05/10/2017. The number of incoming jobs per 5 min is $\lambda(t)$, which is assumed to be poisson distributed with mean equals 1000. each server $k$ can choose a power allocation option $\mu_k^t\in[0,30]$. This option determines the following over the 5 min slot: (1) The electricity money spend of server $k$: $f_k^t(\mu_k^t) = c_k^t\cdot \mu_k^t$, where $c_k^t$ is the per unit EMP of the zone server $k$ belongs to.                                                                                                                                                                                                                                                                                                                           
(2) The number of jobs served $g_k^t(\mu_k^t)$ which follows a Pareto distribution (a.k.a. power law, see \citep{gandhi2012sleep}) of mean $8\log(1+4\mu_k^t)$. (3) Internal budget consumptions $h_k^t\cdot \mu_k^t$, where $h_k^t$ follows a Pareto distribution of mean 5 units. In a typical online service system such as ads service, budget is a measure of internal resources \citep{agarwal2014budget}. The goal is to minimize the total average electricity cost over $T=10000$ slots, i.e. $\sum_{t=1}^T\sum_{k=1}^{50}\expect{c_k^t\cdot \mu_k^t}/T$, subject to the following two requirements: (1) The service rate supports the arrival rate: $\sum_{t=1}^T\sum_{k=1}^{50}\expect{g_k^t(\mu_k^t)}\geq \sum_{t=1}^T\expect{\lambda(t)}$, which is a convex inequality constraint. (2) The internal budget consumption is well-paced, i.e. each cluster consumes a fixed ratio of the total consumed budget in expectation. More specifically, in the simulation, let $\mc I_1,\cdots, \mc I_5$ be index sets of 5 clusters, then, it is required that 
$\sum_{t=1}^T\sum_{k\in \mc I_j}\expect{h_k^t\cdot \mu_k^t} = \beta_j\cdot \sum_{t=1}^T\sum_{k=1}^{50}\expect{h_k^t\cdot \mu_k^t},~~j=1,2,3$ and 
$\sum_{t=1}^T\sum_{k\in \mc I_4\cup\mc I_5}\expect{h_k^t\cdot \mu_k^t} = \beta_4\cdot \sum_{t=1}^T\sum_{k=1}^{50}\expect{h_k^t\cdot \mu_k^t}$, where $[\beta_1,~\beta_2,~\beta_3,~\beta_4] = [0.05,~0.10,~0.25,~0.60]$. In Fig \ref{fig}, we compare our proposed algorithm with the best fixed solution in hindsight choosing the best fixed power allocation knowing all the data, and a benchmark Reac algorithm \citep{gandhi2012sleep}. 
The Reac algorithm is adapted to our pacing scenario by estimating the number of jobs in the next slot via the average of past 10 slots and assign the load according to the pacing ratio. For cluster 4 and cluster 5 (which take up a total ratio of 0.60), the Reac algorithm evenly distribute the workload between the two.  
Our algorithm achieves a similar electricity money spend with the best fixed solution which is better than Reac, while keeping the average number of unserved job low and achieving a fast budget pacing.  

\begin{figure*}[ht!] 
    \centering
    \begin{subfigure}[t]{0.46\textwidth}
        \centering
        \includegraphics[height=5cm] {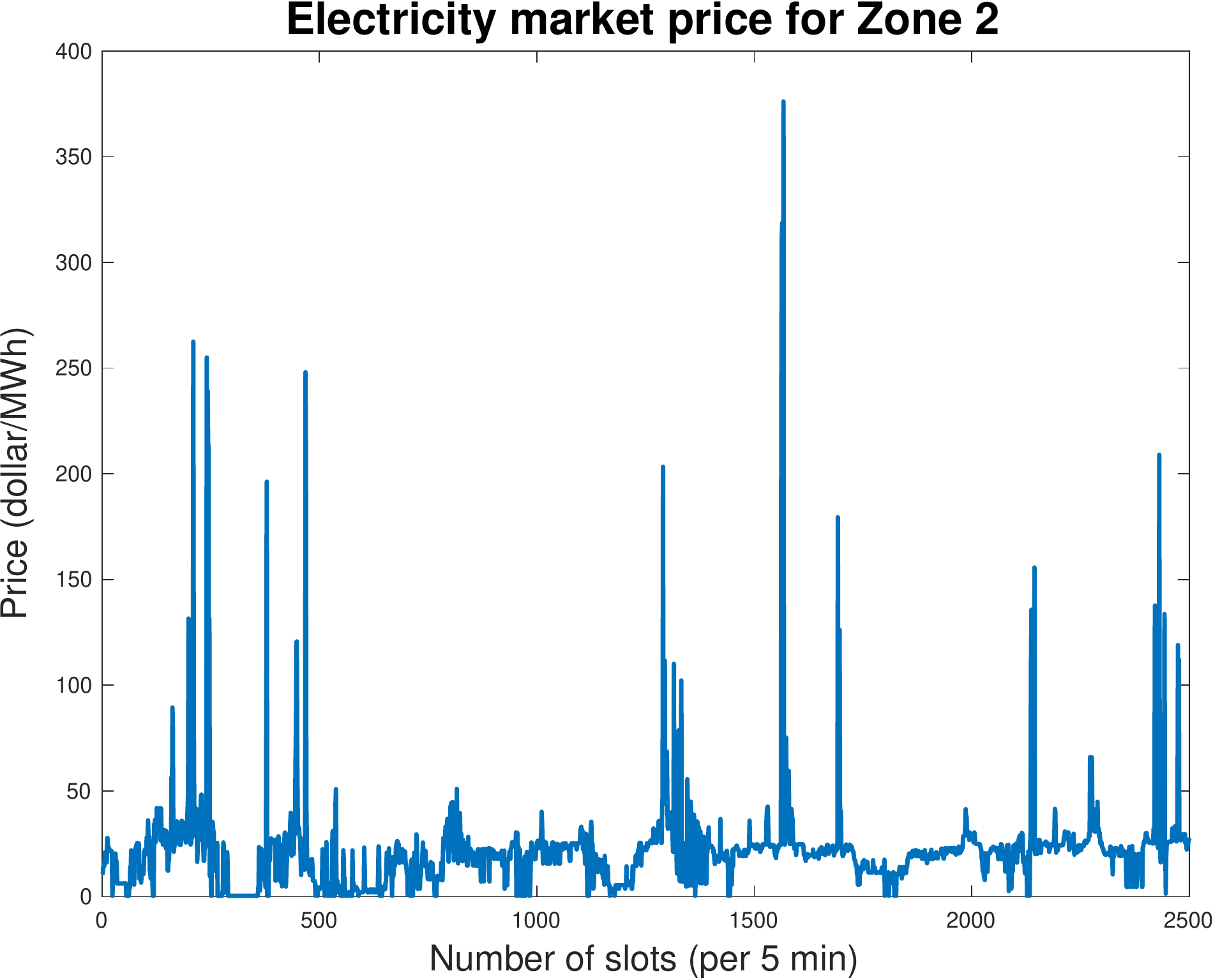}
        \caption{}
    \end{subfigure}%
    ~ 
    \begin{subfigure}[t]{0.46\textwidth}
        \centering
        \includegraphics[height=5cm] {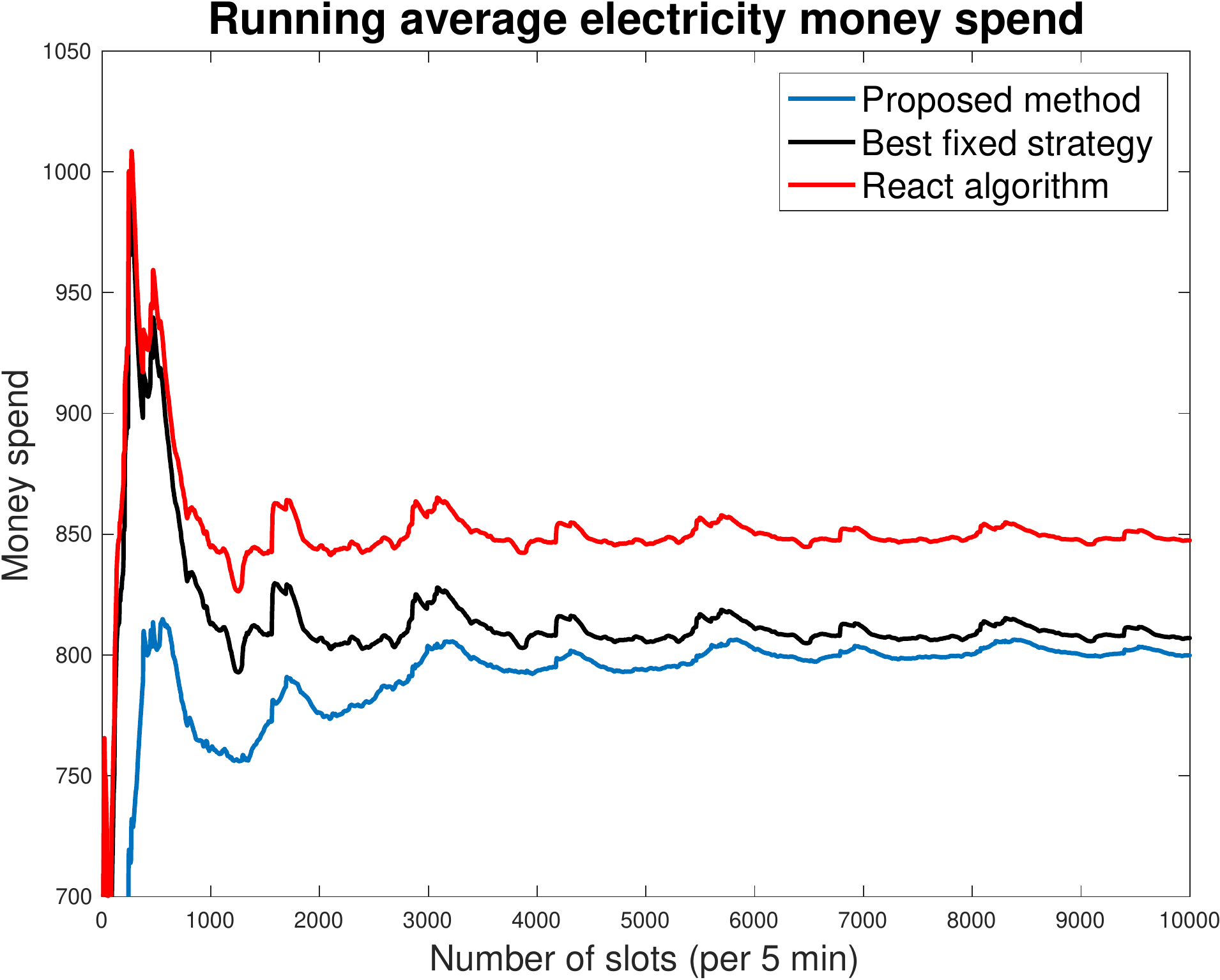}
        \caption{}
    \end{subfigure}
    ~
    \begin{subfigure}[t]{0.46\textwidth}
        \centering
        \includegraphics[height=5cm] {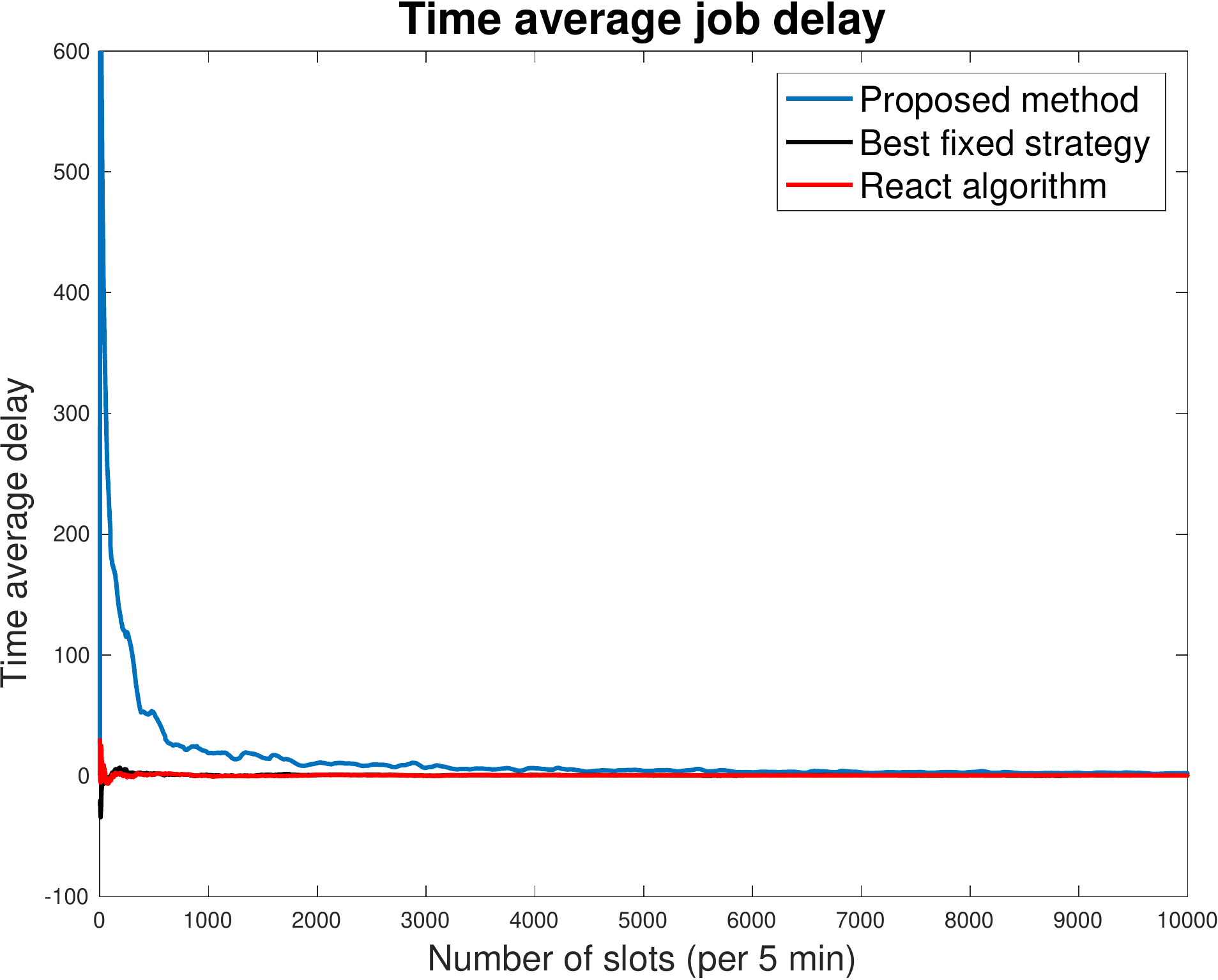}
        \caption{}
    \end{subfigure}
    ~      
    \begin{subfigure}[t]{0.46\textwidth}
        \centering
        \includegraphics[height=5cm] {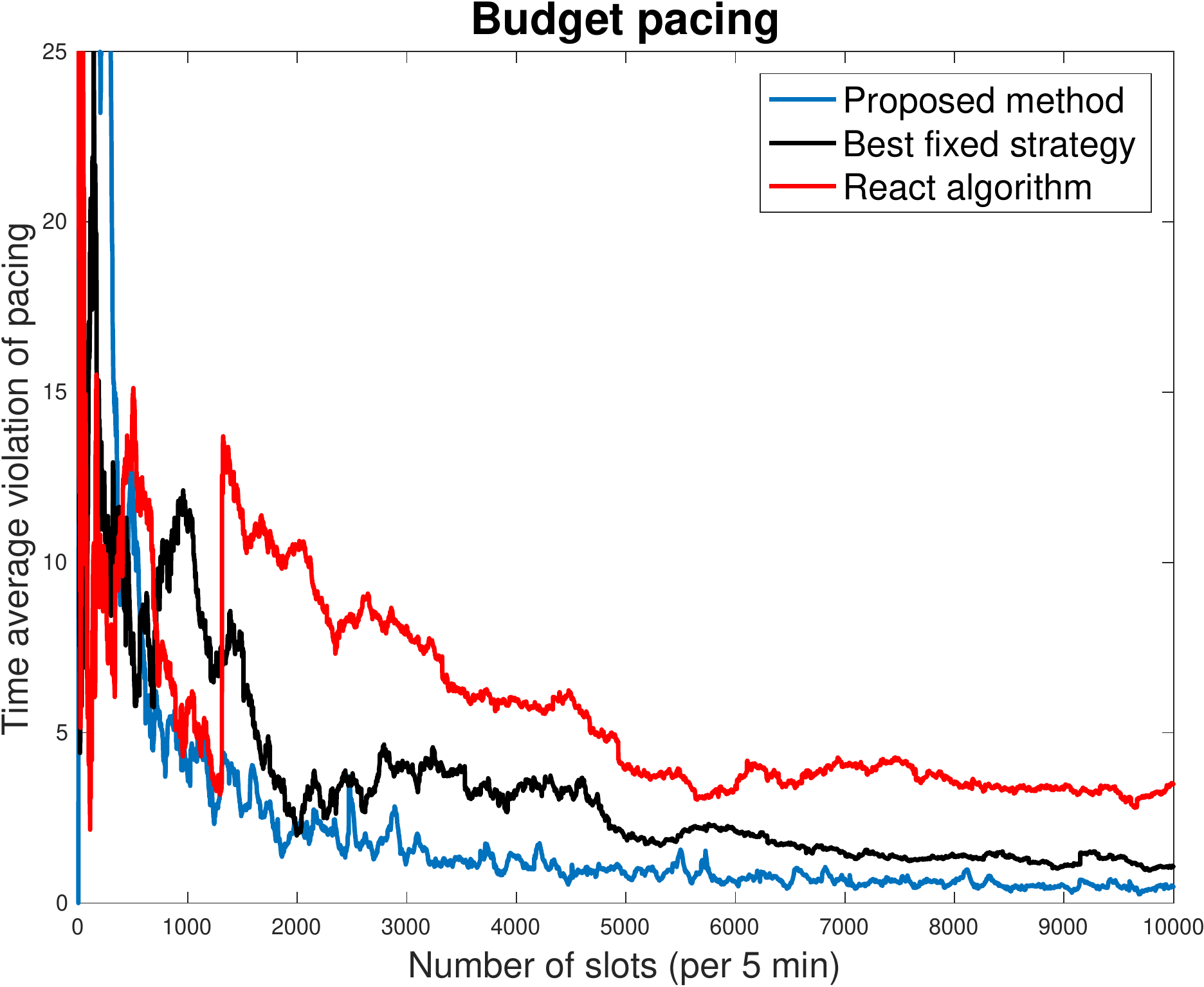}
        \caption{}
    \end{subfigure}   
    \caption{\small{(a) Electricity market prices at zone DUNWOD New York; (b) Average money spent buying electricity; (c) Average unserved jobs. (d) Average violation of pacing constraints.}}
    \label{fig}
\end{figure*}

\section{Conclusions}
This paper proposes a new primal-dual online mirror descent framework for stochastic constrained online learning problem. We introduce a new sequential existence of Lagrange multipliers condition, which is shown to be strictly weaker than the Slater condition, and prove that the proposed algorithm enjoys a $\mathcal{O}(\sqrt{T})$ expected regret and constraint violations. We also obtain an almost dimension free result in the special case when the decision set is a probability simplex.

\acks{This work is supported in part by grant NSF CCF-1718477.}

\bibliographystyle{chicago}
\bibliography{bibliography}

\newpage


\section{Supplement}

\subsection{The pushback property of Bregman divergences}\label{sec:property-divergence}
In this section, we prove the following key property of the Bregman divergence:
\begin{lemma}\label{lem:strong-convex-0}
Let $f:\mathcal{C}\rightarrow\mb{R}$ be a convex function. Fix $\alpha>0$, $y\in\Delta^o$. Suppose $x^*\in \text{argmin}_{x\in\Delta} f(x) + \alpha D(x, y)$ and $x^*\in\Delta^o$, then, for any $z\in\Delta$,
\[
f(x^*) + \alpha D(x^*,y)\leq f(z) + \alpha D(z,y) - \alpha D(z,x^*). 
\]
\end{lemma}

\begin{proof}[Proof of Lemma \ref{lem:strong-convex-0}]
First of all, we recall the following known facts about convex functions and their subgradients whose proofs can be found, for example, in \citep{bertsekas1999nonlinear}:
\begin{itemize}
\item The set $\partial f(x)$ is non-empty for any $x\in\text{int}(\mathcal C)$.
\item For any bounded subset $\mathcal{X}\subseteq\text{int}(\mathcal{C})$, the union $\cup_{x\in\mathcal{X}}\partial f(x)$ is bounded.
\end{itemize}

By definition of Bregman divergence, we have for any $x,~y\in\Delta^o$,
\[
D(x,y) = \omega(x) - \omega(y) - \dotp{\nabla\omega(y)}{x-y},
\]
and
\[
\nabla_{x}D(x,y) = \nabla\omega(x) - \nabla\omega(y).
\]
Now, we claim the following optimality condition:\\
\textbf{Claim 1:} For any $z\in\Delta$, there exists a $\nabla f(x^*)\in\partial f(x^*)$ such that following holds:
\[
\dotp{\nabla f(x^*) + \alpha\nabla\omega(x^*) - \alpha\nabla\omega(y) }{z - x^*}\geq0.
\]
\begin{proof}[Proof of Claim 1]
Fix a constant $h\in(0,1)$. Since $\Delta$ is a convex set, it follows $(1-h)x^*+h z\in\Delta$. Thus, by the fact that $x^*$ is a minimizer:
\begin{align*}
&f(x^*) + \alpha D(x^*,y) \\
\leq& f((1-h)x^*+h z) + \alpha D((1-h)x^*+h z,y)\\
=& f((1-h)x^*+h z) + \alpha \l(D(x^*,y) + \dotp{\nabla D(x^*,y)}{h(z-x^*)} + o(h)\r)\\
=& f((1-h)x^*+h z) + \alpha D(x^*,y)  +  \alpha\l( \dotp{\nabla \omega(x^*) - \nabla \omega (y)}{h(z-x^*)} + o(h) \r),
\end{align*}
where the first equality follows from the fact that $D(x,z)$ is continuously differentially on the first argument at $x=x^*$ with $o(h)$ representing a high order term such that $\lim_{h\rightarrow 0}o(h)/h= 0$, and the second equality follows from the definition of Bregman divergence. Canceling the common term $\alpha D(x^*,y)$ and rearranging the terms give
\begin{equation}\label{eq:sub-grad-bound}
\frac{f((1-h)x^*+h z) - f(x^*)}{h}\geq -\alpha \dotp{\nabla \omega(x^*) - \nabla \omega (y)}{z-x^*} - o(\alpha h)/h.
\end{equation}
Since $f$ is convex and $(1-h)x^*+h z\in \text{int}(\mathcal C),~\forall h<1$, we have for any $\nabla f((1-h)x^*+h z)\in\partial f((1-h)x^*+h z)$.
\begin{equation*}
f(x^*)\geq f((1-h)x^*+h z) + \dotp{\nabla f((1-h)x^*+h z)}{h(x^*-z)}.
\end{equation*}
Substituting this bound into \eqref{eq:sub-grad-bound} gives
\begin{equation}\label{eq:subgradient-2}
\dotp{\nabla f((1-h)x^*+h z)}{z-x^*}\geq -\alpha \dotp{\nabla \omega(x^*) - \nabla \omega (y)}{z-x^*} - o(\alpha h)/h.
\end{equation}
To this point, consider any sequence $\{h_k\}_{k\geq0}\subseteq(0,1)$ such that $\lim_{k\rightarrow\infty}h_k = 0$. By the aforementioned property of subgradient, we have the union $\cup_{k\geq0}\partial f((1-h_k)x^*+h_k z)$ is bounded. Thus, the sequence
$\{\nabla f((1-h_k)x^*+h_k z) \}_{k\geq0}$ is bounded, and there exists a subsequence $\{\nabla f((1-h_{k_{\ell}})x^*+h_{k_\ell} z) \}_{\ell\geq0}$ such that $\nabla f((1-h_{k_{\ell}})x^*+h_{k_\ell} z) \rightarrow d$. On the other hand, by definition of subgradient, we have for any $u\in\mathcal{C}$,
\[
f(u)\geq f((1-h_{k_{\ell}})x^*+h_{k_\ell} z) + \dotp{\nabla f((1-h_{k_{\ell}})x^*+h_{k_\ell} z)}{u - ((1-h_{k_{\ell}})x^*+h_{k_\ell} z)}.
\]
Taking the limit $\ell\rightarrow\infty$ gives
\[
f(u)\geq f(x^*) + \dotp{d}{u - x^*},
\]
where we use the fact that a convex function must be continuous on the interior point $x^*$ of $\mathcal{C}$. This implies that 
$d\in\partial f(x^*)$. Substituting $\{h_{k_\ell}\}_{\ell\geq0}$ into \eqref{eq:subgradient-2} and taking the limit finish the proof.
\end{proof}

Thus, by Claim 1, we have there exists a $\nabla f(x^*)$,
\begin{align*}
&\alpha (D(z,y) - D(z,x^*)) \\
=&\alpha\l(\omega(z) - \omega(y) - \dotp{\nabla\omega(y)}{z - y}\r) 
-\alpha\l(\omega(z) - \omega(x^*) - \dotp{\nabla\omega(x^*)}{z - x^*}\r) \\
=&\alpha\l( \omega(x^*) - \omega(y)  + \dotp{\nabla\omega(x^*) }{z - x^*} - \dotp{\nabla\omega(y)}{z - y}\r)\\
=&\alpha\l( \omega(x^*) - \omega(y)  + \dotp{\nabla f(x^*)/\alpha + \nabla\omega(x^*) - \nabla\omega(y)}{z - x^*} - \dotp{\nabla\omega(y)}{z - y}\r)\\
&-\dotp{\nabla f(x^*)}{z - x^*} + \alpha\dotp{\nabla\omega(y)}{z - x^*} \\
\geq&\alpha\l( \omega(x^*) - \omega(y)   - \dotp{\nabla\omega(y)}{x^* - y}\r)
- \dotp{\nabla f(x^*)}{z - x^*} \\
=& \alpha D(x^*,y) - \dotp{\nabla f(x^*)}{z - x^*}\\
\geq& \alpha D(x^*,y)  + f(x^*) - f(z),
\end{align*}
where third equality follows from adding and subtracting $\dotp{\nabla f(x^*)}{z - x^*} - \alpha\dotp{\nabla\omega(y)}{z - x^*}$, 
 the first inequality follows from the aforementioned optimality condition and the last inequality follows from convexity that $f(z)\geq f(x^*) + \dotp{\nabla f(x^*)}{z - x^*}$. Rearranging the terms yields the desired result.
\end{proof}

\subsection{SELM and constraint qualifications}\label{sec:constraint-qualification}
\subsubsection{Slater condition implies SELM}
The SELM assumption is actually implied by the Slater condition. More specifically, Slater condition considers the scenario where there is no equality constraint and there exists a $\mu\in\Delta$ such that $\overline{g}_i(\mu)< 0,~\forall i\in\{1,2,\cdots,L\}$. 
First of all, it is well-known that the Slater condition is
sufficient for the existence of a dual optimal solution (see, for example, \citep{bertsekas1999nonlinear}). Furthermore, the following lemma, which is essentially the same as Lemma 1 of \citep{nedic2009approximate}, implies that the set of dual optimal solutions is also bounded:
\begin{lemma}\label{lem:slater-selm}
Consider the convex program \eqref{eq:partial-static-prob} without equality constraints $\overline{\mathbf{h}}(\mu) =0$, and define the Lagrange dual function $q^{(t,k)}(\lambda)=\inf_{\mu\in\Delta}\l\{ \overline{f}^{(t,k)}(\mu) + \sum_{i=1}^m\lambda_i\overline{g}_i(\mu) \r\}$. Suppose there exists $\widetilde{\mu}\in\Delta$ such that $\overline g_i\l(\widetilde{\mu}\r)\leq-\varepsilon,~\forall i\in\{1,2,\cdots,L\}$ for some positive constant $\varepsilon>0$. Then, the level set 
$\mathcal{V}_{\bar\lambda}=\l\{ \lambda_1,\lambda_2,\cdots,\lambda_L\geq0,~ q^{(t,k)}(\lambda)\geq q^{(t,k)}(\bar\lambda)\r\}$ is bounded for any nonnegative $\bar\lambda$. Furthermore, we have
$\max_{\lambda\in\mathcal{V}_{\bar\lambda}}  \|\lambda\|_2\leq\varepsilon^{-1}\l(  \overline{f}^{(t,k)}(\widetilde{\mu})-q^{(t,k)}(\bar\lambda)   \r)$.
\end{lemma}
Note that since $|f^t(\mu)|$ is bounded by some constant $F>0$ as stated in Assumption \ref{as:basic}. Taking 
$\bar\lambda = \lambda^*$ for any optimal dual solution $\lambda^*$, and notice that 
$\overline{f}^{(t,k)}(\widetilde{\mu})\leq F$, $q^{(t,k)}(\lambda^*)\geq\min_{\mu\in\Delta}\overline{f}^{(t,k)}(\mu)\geq-F$,
the above lemma readily implies
$\max_{\lambda\in\mathcal{V}^*}\|\lambda\|_2\leq 2F/\varepsilon$. Thus, Slater condition implies the existence of Lagrange multiplier condition.

\subsubsection{SELM is equivalent to Mangasarian-Fromovitz constraint qualification (MFCQ)}
In this section, we show SELM is able to handle general equality constraints and thus strictly weaker than the Slater condition.
In 1977, J. Gauvin \citep{gauvin1977necessary} observed that for any constrained convex program,
where both the objective and constraint functions are continuously differentiable,
the Mangasarian-Fromovitz constraint qualification (MFCQ) condition is in fact equivalent to the boundedness of the set of Lagrange multipliers.\footnote{In fact, MFCQ does not require convexity of the constrained programs. Thus, the result in \citep{gauvin1977necessary} even applies to non-convex programs.} More specifically, MFCQ 
 is defined as follows:

\begin{definition}[Mangasarian-Fromovitz constraint qualification (MFCQ)]
Consider a convex program:
\begin{equation}\label{eq:general-program}
\min_{x\in\mathbb{R}^d}~f(x),~~s.t. ~~g_i(x)\leq 0,~i\in\{1,2,\cdots,L\},~~\dotp{ h_j}{ x}=b_j,
~j\in\{1,2,\cdots,M\}.
\end{equation} 
It satisfies MFCQ if (a) The vectors $\{ h_j\}_{j=1}^M$ are linearly independent. (b) For a solution $ x^*$ to the above program, there exists some ${y}\in\mathbb{R}^d$ such that $\dotp{\nabla g_i( x^*)}{y}<0,~\forall i\in I( x^*)$, where 
$I( x^*) = \{i~| ~g_i(x^*) = 0\}$. 
\end{definition}
\begin{theorem}[\citep{gauvin1977necessary}]\label{theorem:mfcq}
Consider the Karush-Kuhn-Tucker(KKT) set of the program \eqref{eq:general-program}, which is the set $K( x^*)$ of vectors 
$(\lambda,\eta)\in\mathbb{R}^L_+\times \mathbb{R}^M$ such that the following set of equations holds:
\begin{align*}
-\nabla f( x^*) = \sum_{i=1}^L\lambda_i\nabla g_i( x^*) + \sum_{j=1}^M\eta_j \mf h_j,~\lambda\geq0,
~\lambda_ig_i(x^*) = 0,~\forall i\in\{1,2,\cdots,M\}.
\end{align*}
Then, the set $K( x^*)$ is non-empty and bounded if and only if MFCQ is satisfied for \eqref{eq:general-program}.
\end{theorem}

Note that compared to \eqref{eq:general-program} our program \eqref{eq:partial-static-prob} has an extra set constraint 
$\mu\in\Delta$. The good news is that for the case where $\Delta$ is a probability simplex, i.e. it can be written explicitly as $\{\mu\in\mathbb R^d:~\mu_i\geq0,~\forall i,~\sum_{i=1}^d\mu_i=1\}$, 
applying Theorem \ref{theorem:mfcq}, we have the following lemma whose proof is delayed to Section \ref{sec:supp-lemma-proof}:
\begin{lemma}\label{cor:mfcq}
Consider the optimization problem \eqref{eq:partial-static-prob} for any time slot $t$ and any time period $k$ where $\Delta$ is the probability simplex. 
Suppose (a) The vectors $\{\mathbf{1},~\expect{h^t_1},~\expect{h^t_2},~\cdots,~\expect{h^t_M}\}$ are linearly independent. (b) There exists a solution to \eqref{eq:partial-static-prob}, denoted as $\mu^*$, and a vector $ y\in\mathbb{R}^d$ such that 
$\dotp{\nabla \overline g_i(\mu^*)}{ y}<0,~\forall i\in I(\mu^*)$, where 
$I(\mu^*) = \{i~| ~\overline{g}_i(\mu^*) = 0\}$. Then, the set of Lagrange multipliers $\mathcal{V}^* := \text{argmax}_{\lambda\in\mathbb{R}^L_+,~\eta\in\mathbb{R}^M}q^{(t,k)}(\lambda,\eta)$, where $q^{(t,k)}$ is defined in \eqref{eq:target-dual}, is non-empty and bounded.
\end{lemma}

\begin{remark}
In the case where there is no inequality constraints in \eqref{eq:partial-static-prob}, lemma \ref{cor:mfcq} gives a simple objective-irrelevant equivalence condition of SELM that $\{\mathbf{1},~\expect{h^t_1},~\expect{h^t_2},~\cdots,~\expect{h^t_M}\}$ are linearly independent, which could be useful for online linear program. 
\end{remark}

For general scenarios where $\Delta$ is just an arbitrary abstract convex set, we have the following definition of generalized MFCQ following \citep{nguyen1980conditions}. First, we have the definitions of normal cones and tangent cones:
\begin{definition}[Normal cone]
Consider any set $S\subseteq\mb R^d$, the normal cone of $S$ at any $x\in S$ is
\[
N(S,x) := \{g\in\mb R^d:~\dotp{g}{x-y}\geq0,~\forall y\in \mb R^d\}.
\]
\end{definition}
Note that normal cone at $x\in S$ is the subgradient of the indicator function of $S$, namely $I_S(x)$. To see this, consider any $y\in \mb R^d$, then, we have $g$ is a subgradient of $I_S(x)$ at $x$ if
\[
I_S(y)\geq I_S(x) + \dotp{g}{y-x},~~\forall y\in\mb R^d.
\]
Note that if $y\not\in S$, then $I_S(y)=+\infty$, otherwise, $I_S(y) = I_S(x) = 0$. Thus,  $\dotp{g}{x-y}\geq0$.

\begin{definition}[Tangent cone]
Consider any set $S\subseteq\mb R^d$, the tangent cone of $S$ at any $x\in S$ is
\[
T(S,x) := \text{cone}(S - x) = \{\lambda d:~\lambda\geq0,~d\in S-x\},
\]
and $S-x = \{y\in\mb R^d, y = z-x,~ \exists z\in S\}$.
\end{definition}

\begin{definition}[Generalized MFCQ]
Consider a convex program:
\begin{equation}\label{eq:general-program-2}
\min_{\mf x\in S}~f(\mf x),~~s.t. ~~g_i(\mf x)\leq 0,~i\in\{1,2,\cdots,L\},~~\dotp{\mf h_j}{\mf x}=b_j,
~j\in\{1,2,\cdots,M\}.
\end{equation} 
It satisfies the generalized MFCQ if (a) The vectors $\{\mf h_j\}_{j=1}^M$ are linearly independent. (b) For a solution $\mf x^*$ to the above program, there exists some $y\in\text{int}(T(S,x^*))$ such that $\dotp{\nabla g_i( x^*)}{ y}<0,~\forall i\in I( x^*)$ and any subgradient $\nabla g_i( x^*)$, where $I( x^*) = \{i~| ~g_i(\mf x^*) = 0\}$ and $\text{int}(T(S,x^*))$ denotes the interior of $T(S,x^*)$. 
\end{definition}
Note that this definition requires the interior of $T(S,x^*)$ to be non-empty, which \textit{does not} work for the case where $S$ is a probability simplex. This is why we have a separate lemma (Lemma \ref{cor:mfcq}). When assuming the interior of $T(S,x^*)$ is non-empty, we have the following theorem:
\begin{theorem}[\citep{nguyen1980conditions}]
Consider the Karush-Kuhn-Tucker(KKT) set of the program \eqref{eq:general-program-2}, which is the set $K( x^*)$ of vectors 
$(\lambda,\eta)\in\mathbb{R}^L_+\times \mathbb{R}^M$ such that the following set of equations holds:
\begin{align*}
0\in \partial f( x^*) + \sum_{i=1}^L\lambda_i\nabla g_i( x^*) + \sum_{j=1}^M\eta_j \mf h_j + N(S,x^*),~\lambda\geq0,
~\lambda_ig_i(x^*) = 0,~\forall i\in\{1,2,\cdots,M\}.
\end{align*}
Then, the set $K( x^*)$ is non-empty and bounded if and only if  \eqref{eq:general-program} satisfies the generalized  MFCQ.
\end{theorem}
Applying the above theorem to \eqref{eq:partial-static-prob} with $S=\Delta$, we readily get the equivalence condition for the existence and boundedness of Lagrange multipliers for \eqref{eq:partial-static-prob} as follows
\begin{corollary}\label{cor:mfcq-2}
Consider the optimization problem \eqref{eq:partial-static-prob} for any time slot $t$ and any time period $k$ where $\Delta$ has an nonempty interior. 
Suppose (a) The vectors $\{\expect{h^t_1},~\expect{h^t_2},~\cdots,~\expect{h^t_M}\}$ are linearly independent. (b) There exists a solution to \eqref{eq:partial-static-prob}, denoted as $\mu^*$, and a vector $ y\in \text{int}(T(\Delta,\mu^*))$ such that 
$\dotp{\nabla \overline g_i(\mu^*)}{ y}<0,~\forall i\in I(\mu^*)$, where 
$I(\mu^*) = \{i~| ~\overline{g}_i(\mu^*) = 0\}$. Then, the set of Lagrange multipliers $\mathcal{V}^* := \text{argmax}_{\lambda\in\mathbb{R}^L_+,~\eta\in\mathbb{R}^M}q^{(t,k)}(\lambda,\eta)$, where $q^{(t,k)}$ is defined in \eqref{eq:target-dual}, is non-empty and bounded.
\end{corollary}

\subsubsection{SELM implies weak EBC}\label{sec:selm-ebc}
In this section, we prove a key property of SELM, namely Lemma \ref{lem:leb}, which says SELM implies a weak EBC condition. We restate the lemma as follows, and for simplicity, we omit the subscript $t,k$ on the set $\mathcal{V}^*$ for simplicity:
\begin{lemma}\label{lem:leb-0}
Suppose Assumption \ref{as:selm} holds, then, there exists constants $c_0,\ell_0>0$ such that the dual function $q^{(t,k)}(\lambda,\eta)$ defined in \eqref{eq:target-dual} satisfies a weak error bound condition, namely, for any  $(\lambda^*,\eta^*) \in \mathcal{V}^*$, 
$q^{(t,k)}(\lambda^*,\eta^*) - q^{(t,k)}(\lambda,\eta)\geq c_0 \cdot\text{dist}((\lambda,\eta),\mathcal{V}^*)$ for any 
$(\lambda,\eta)$ such that
$\text{dist}((\lambda,\eta),\mathcal{V}^*)\geq\ell_0$.
\end{lemma}

\begin{proof}[Proof of Lemma \ref{lem:leb-0}]
	
Since $\mathcal{V}^*$ is bounded, there must exist $\ell_0 >0$ such that $\mathcal{S}_1 := \{(\lambda,\eta):dist((\lambda,\eta),\mathcal{V}^*) = l_0\}\neq \emptyset$. Define $\tilde q: =\sup_{(\lambda,\eta)\in\mathcal{S}_1}q^{(t,k)}(\lambda,\eta)$. Then, since 
the set $\mathcal{S}_1$ is closed, there exists some constant $c_0>0$ such that 
$q^{(t,k)}(\lambda^*,\eta^*) - \tilde q\geq c_0 l_0$.
Now, consider any $(\lambda,\eta)$ such that $\text{dist}((\lambda,\eta),\mathcal{V}^*)\geq l_0$, and choose 
$(\lambda^*,\eta^*)\in \mathcal{V}^*$ such that 
\begin{equation}\label{eq:min-attain}
(\lambda^*,\eta^*) = \text{argmin}_{(\lambda_0,\eta_0)\in\mathcal{V}^*}\|(\lambda_0,\eta_0) - (\lambda,\eta)\|_2^2,
\end{equation}
i.e. $\|(\lambda^*,\eta^*) - (\lambda,\eta)\|_2 = \text{dist}((\lambda,\eta),\mathcal{V}^*) \geq l_0$.

 Choose $\theta := \frac{l_0}{\|(\lambda^*,\eta^*) - (\lambda,\eta)\|_2}$. Note that $0 < \theta \leq 1$.  Let $(\tilde\lambda, \tilde\eta):=((1-\theta)\lambda^*+\theta \lambda, (1-\theta) \eta^* + \theta \eta)$. The next claim shows that $(\tilde\lambda, \tilde\eta) \in \mathcal{S}_1$.
	
\textbf{Claim 1:} $(\tilde\lambda, \tilde\eta) \in \mathcal{S}_1$.
\begin{proof}
	It is easy to verify that $\|(\tilde\lambda,\tilde\eta) - (\lambda^*,\eta^*)\|_2 = l_0$. To prove this claim, it suffices to show that   
	\[
	(\lambda^*,\eta^*) = \text{argmin}_{(\lambda_0,\eta_0)\in\mathcal{V}^*}\|(\tilde\lambda,\tilde\eta) - (\lambda_0,\eta_0)\|_2^2.
	\]
	To see this, suppose on the contrary, there exists $(\overline\lambda,\overline\eta)\neq (\lambda^*,\eta^*)$ such that 
	$(\overline\lambda,\overline\eta)$ attains the above minimum, then, by the strong convexity of the square norm function and convexity of the set $\mathcal{V}^*$, the solution is unique, and it follows
	\begin{multline*}
	\|(\overline\lambda,\overline\eta) - (\lambda,\eta)\|_2
	\leq  \|(\overline\lambda,\overline\eta) - (\lambda',\eta')\|_2 + \| (\lambda',\eta') -  (\lambda,\eta)\|_2\\
	< \|(\lambda^*,\eta^*) - (\lambda',\eta')\|_2 + \| (\lambda',\eta') -  (\lambda,\eta)\|_2
	= \|(\lambda^*,\eta^*) - (\lambda,\eta)\|_2,
	\end{multline*}
	where the strict inequality follows from the aforementioned strong convexity and the last equality follows from the fact that 
	$(\lambda',\eta')\in\mathcal{L}$. However, this implies $\overline\lambda,\overline\eta$ is of smaller distance to $(\lambda,\eta)$ contradicting \eqref{eq:min-attain}. 
\end{proof}
	
By the concavity of $q^{(t,k)}(\lambda,\eta)$, we have,
\begin{equation}\label{eq:convexity-ineq}
q^{(t,k)}((1-\theta)\lambda^*+\theta \lambda, (1-\theta) \eta^* + \theta \eta) \geq (1-\theta)q^{(t,k)}(\lambda^*,\eta^*) + \theta q^{(t,k)}(\lambda,\eta).
\end{equation}
This further implies that 
\begin{align*}
~&q^{(t,k)}(\tilde\lambda, \tilde\eta) \geq (1-\theta)q^{(t,k)}(\lambda^*,\eta^*) + \theta q^{(t,k)}(\lambda,\eta)\\
\Rightarrow~&q^{(t,k)}(\tilde\lambda, \tilde\eta) - q^{(t,k)}(\lambda^*,\eta^*) \geq \theta
(q^{(t,k)}(\lambda,\eta) - q^{(t,k)}(\lambda^*,\eta^*)).
\end{align*}

Recalling the definition of $\tilde q =\sup_{(\lambda,\eta)\in\mathcal{S}_1}q^{(t,k)}(\lambda,\eta)$ and that $(\tilde\lambda, \tilde\eta) \in \mathcal{S}_1$ by Claim 1, we have 
\begin{align*}
~&\tilde q- q^{(t,k)}(\lambda^*,\eta^*)\geq \theta(q^{(t,k)}(\lambda,\eta) - q^{(t,k)}(\lambda^*,\eta^*))\\
\Rightarrow~&q^{(t,k)}(\lambda^*,\eta^*) - q^{(t,k)}(\lambda,\eta)\geq\frac{1}{\theta}
(q^{(t,k)}(\lambda^*,\eta^*) - \tilde q).
\end{align*}
Recalling that $q^{(t,k)}(\lambda^*,\eta^*) - \tilde q\geq c_0l_0$ and $\theta = \frac{l_0}{\|(\lambda^*,\eta^*) - (\lambda,\eta)\|_2}$, we have
\[
q^{(t,k)}(\lambda^*,\eta^*) - q^{(t,k)}(\lambda,\eta)\geq c_0\text{dist}((\lambda,\eta),\mathcal{V}^*),
\]
and we finish the proof.

\end{proof}

\subsection{On the relation between weak EBC and classical EBC}
Recall that the classical EBC, which has been shown to accelerate the convergence rate solving unconstrained and constrained programs \citep{tseng2010approximation, yang2015rsg, xu2017admm, wei2018solving}, is stated as follows:
\begin{definition}\label{def:local-error}
Let $F(\mathbf{x})$ be a convex function over $\mathbf{x}\in\mathcal{X}$.
Suppose $\Lambda^*:= \argmin_{\mathbf{x}\in\mathcal{X}}F(\mathbf{x})$ is non-empty. The function $F(\mathbf{x})$ is said to satisfy the error bound condition (EBC) with parameters $\beta\in(0,1], \delta >0$ and $C_\delta>0$ if for any $\mf x\in\mc S_\delta$, the $\delta$-sublevel set defined as $\{\mf x\in \mathcal{X}~|~F(\mathbf{x}) - F(\mathbf{x}^*)\leq\delta,~\mf x^*\in\Lambda^*\}$,
\begin{equation}\label{eq:local-error-bound}
\text{dist}(\mf x,\Lambda^*)\leq C_\delta (F(\mf x) - F(\mf x^*))^\beta,
\end{equation}
where $C_\delta$ is a positive constant possibly depending on $\delta$. In particular, when $\beta = 1/2$, $F(\mf x)$ is said to be locally quadratic and when $\beta = 1$, it is said to be locally linear.
\end{definition}

The following lemma shows that if the dual function further satisfies classical EBC, then, we can show that weak EBC holds with computable constants $\ell_0, c_0>0$.
\begin{lemma}\label{lem:ebc-estimate}
Suppose Assumption \ref{as:selm} holds, the dual function $q^{(t,k)}(\lambda,\eta)$ is continuous and satisfies an EBC as is defined in Definition \ref{def:local-error}, then, one has for any $(\lambda,\eta)\in\mathbb{R}^L_+\times \mathbb{R}^M$ such that 
$\text{dist}((\lambda,\eta),\mathcal{V}^*)\geq C_\delta\delta^\beta$,
\[
\text{dist}((\lambda,\eta),\mathcal{V}^*) \leq C_\delta\delta^{\beta-1}
(q^{(t,k)}(\lambda^*,\eta^*) - q^{(t,k)}(\lambda,\eta)),~\forall \lambda\in\mathbb{R}^L_+,~\eta\in\mathbb{R}^M
\]
\end{lemma}
The proof of this lemma is delayed to Section \ref{sec:supp-lemma-proof}.

\subsection{Supporting lemmas in proof of Theorem \ref{main:theorem}}\label{proof:thm}
Throughout the section, we let $\mathcal{F}_t$ be the system history up to time $t$, which includes 
$\{g_i^{\tau}\}_{\tau=0}^{t-1}$,~$\{h_i^{\tau}\}_{\tau=0}^{t-1}$,~and $\{f^\tau\}_{\tau=0}^{t-1}$.

\subsubsection{Proof of lemmas in Section \ref{sec:reg-analysis}}\label{sec:pf-regret}


\begin{proof}[Proof of Lemma \ref{lem:strong-convex-queue}]
Applying Lemma \ref{lem:strong-convex} by setting $y = \mu^{t-1}$, $x^* = \mu^t$, $f(x) =\dotp{x}{p}$ and
$$
p = V\nabla f^{t-1}(\mu^{t-1}) + \sum_{i=1}^{L} Q_i(t)\nabla g_i^{t-1}(\mu^{t-1}) 
 + \sum_{i=1}^{M} H_i(t)h_i^{t-1},
$$ 
we have
\begin{multline}\label{eq:interim-0}
 \dotp{V\nabla f^{t-1}(\mu^{t-1}) + \sum_{i=1}^{L} Q_i(t)\nabla g_i^{t-1}(\mu^{t-1}) 
 + \sum_{i=1}^{M} H_i(t)h_i^{t-1}}{\mu^t} + \alpha D(\mu^{t},{\mu}^{t-1}) \\
 \leq
  \dotp{V\nabla f^{t-1}(\mu^{t-1}) + \sum_{i=1}^{L} Q_i(t)\nabla g_i^{t-1}(\mu^{t-1}) 
 + \sum_{i=1}^{M} H_i(t)h_i^{t-1}}{\mu} + \alpha (D(\mu,\mu^{t-1}) - D(\mu,\mu^t))
\end{multline}

On the other hand, define 
$$\tilde{g}_i^t := g^{t-1}_{i}(\mu^{t-1}) + \dotp{\nabla g_{i}^{t-1} (\mu^{t-1})}{\mu^t-\mu^{t-1}}.$$
 Using the updating rule (\ref{eq:Q-update}), (\ref{eq:H-update}) and Holder's inequality that $\dotp{x}{y}\leq \|x\|\|y\|_*$, we have
 \begin{align*}
 H_i(t+1)^2 - H_i(t)^2 &=  2H_i(t)(\dotp{h_i^{t-1}}{\mu^t}-b_i) +  |\dotp{h_i^{t-1}}{\mu^t}-b_i|^2\\
 &\leq  2H_i(t)(\dotp{h_i^{t-1}}{\mu^t}-b_i) + \frac{8R}{\beta}\|h_i^{t-1}\|_*^2,\\
 Q_i(t+1)^2 - Q_i(t)^2 &= \max\{Q_i(t) + \tilde{g}_i^t,~0\}^2 - Q_i(t)^2\leq 2Q_i(t)\tilde{g}_i^t+  (\tilde{g}_i^t)^2\\
 &\leq 2Q_i(t)\tilde{g}_i^t+  2(g^{t-1}_{i}(\mu^{t-1}))^2 + \frac{4R}{\beta}\|\nabla g_{i}^{t-1} (\mu^{t-1})\|_*^2,
 \end{align*}
where the inequality for $H_i(t+1)^2 - H_i(t)^2$ follows from 
\[
|\dotp{h_i^{t-1}}{\mu^t}-b_i|^2\leq 2 |\dotp{h_i^{t-1}}{\mu^t}|^2 + 2|b_i|^2 = 
2 |\dotp{h_i^{t-1}}{\mu^t}|^2 + 2|\expect{\dotp{h_i^{t-1}}{\mu^*}}|^2 \leq 8R/\beta,
\]
via Assumption \ref{as:basic}(3) that $\sup_{\mu^a,\mu^b\in\Delta}\|\mu^a-\mu^b\|^2\leq 2R/\beta$ and $b_i = \expect{\dotp{h_i^{t-1}}{\mu^*}}$.
The first inequality in the bound on $Q_i(t+1)^2 - Q_i(t)^2$ follows from the fact that if $Q_i(t) + \tilde{g}_i^t\geq0$, then, the equality is attained and if $Q_i(t) + \tilde{g}_i^t\leq0$, 
 $Q_i(t) + \tilde{g}_i^t\leq0$, then, 
$\max\{Q_i(t) + \tilde{g}_i^t,~0\}^2 = 0$ while  $ Q_i(t)^2 + 2Q_i(t)\tilde{g}_i^t+  (\tilde{g}_i^t)^2\geq0$. The third line follows from 
$\sup_{\mu^a,\mu^b\in\Delta}\|\mu^a-\mu^b\|^2\leq 2R/\beta$. 
Thus, we have
\begin{multline}\label{eq:q-bound}
\Delta(t)
 \leq \sum_{i=1}^LQ_i(t)\tilde{g}_i^t + \sum_{i=1}^MH_i(t)(\dotp{h_i^t}{\mu^t}-b_i) \\
+ \frac{4R}{\beta}\sum_{i=1}^M\|h_i^t\|_{*}^2 + \sum_{i=1}^L(g^{t-1}_{i}(\mu^{t-1}))^2 +  \frac{2R}{\beta} \sum_{i=1}^L\|\nabla g_{i}^{t-1} (\mu^{t-1})\|_*^2\\
\leq  \sum_{i=1}^LQ_i(t)\tilde{g}_i^t + \sum_{i=1}^MH_i(t)(\dotp{h_i^t}{\mu^t}-b_i)  + \frac{4RH^2}{\beta} + G^2+\frac{2RD_2^2}{\beta},
\end{multline}
where the last inequality follows from Assumption \ref{as:basic}(1). 
To this point, we consider the following ``drift-plus-penalty'' term, i.e.
\begin{align*}
&V\dotp{\nabla f^{t-1}(\mu^{t-1})}{\mu^t-\mu^{t-1}}+\Delta(t) + \alpha D(\mu^t, \mu^{t-1})\\
\leq&  V\dotp{\nabla f^{t-1}(\mu^{t-1})}{\mu^t-\mu^{t-1}} + \sum_{i=1}^LQ_i(t)\l(g^{t-1}_{i}(\mu^{t-1}) + \dotp{\nabla g_{i}^{t-1} (\mu^{t-1})}{\mu^t-\mu^{t-1}}\r)\\
 &+ \sum_{j=1}^{M} H_j(t)(\dotp{h_j^{t-1}}{\mu^t}-b_i)+ \alpha D(\mu^t,\mu^{t-1})+ \frac{4RH^2}{\beta} + G^2+\frac{2RD_2^2}{\beta}.
\end{align*}
where the first inequality follows from \eqref{eq:q-bound}.  
Now, by \eqref{eq:interim-0}, we have for any $\mu\in\Delta$,
\begin{align*}
&V\dotp{\nabla f^{t-1}(\mu^{t-1})}{\mu^t-\mu^{t-1}}+\Delta(t) + \alpha D(\mu^t, \tilde{\mu}^{t-1})  \nonumber\\
 \leq&  V\dotp{\nabla f^{t-1}(\mu^{t-1})}{\mu-\mu^{t-1}} + \sum_{i=1}^LQ_i(t)\l(g^{t-1}_{i}(\mu^{t-1}) + \dotp{\nabla g_{i}^{t-1} (\mu^{t-1})}{\mu-\mu^{t-1}}\r)  \nonumber\\
 &+ \sum_{j=1}^{M} H_j(t)(\dotp{h_j^{t-1}}{\mu}-b_i)+ \alpha D(\mu,\mu^{t-1}) - \alpha D(\mu, \mu^t)
 + \frac{4RH^2}{\beta} + G^2+\frac{2RD_2^2}{\beta}.  
\end{align*}

Note that by convexity, we have for any $\mu$,
\begin{align*}
f^{t-1}(\mu) &\geq f^{t-1}(\mu^{t-1}) + \dotp{\nabla f^{t-1}(\mu^{t-1})}{\mu - \mu^{t-1}},\\
g_i^{t-1}(\mu) &\geq g_i^{t-1}(\mu^{t-1}) + \dotp{\nabla g_i^{t-1}(\mu^{t-1})}{\mu - \mu^{t-1}}. 
\end{align*}
Thus, it follows \eqref{eq:interim-1} holds.
\end{proof}

\begin{proof}[Proof of Lemma \ref{lem:lhs-lower-bound}]
We have
\begin{align*}
&V\dotp{\nabla f^{t-1}(\mu^{t-1})}{\mu^t-\mu^{t-1}} + \alpha D(\mu^t, \mu^{t-1}) \\
\geq&V\dotp{\nabla f^{t-1}(\mu^{t-1})}{\mu^t-\mu^{t-1}} + \frac{\alpha\beta}{2} \|\mu^t- \mu^{t-1}\|^2 \\
\geq&-V \|\nabla f^{t-1}(\mu^{t-1})\|_{*} \|\mu^t-\mu^{t-1}\| + \frac{\alpha\beta}{2} \|\mu^t- \mu^{t-1}\|^2\\
\geq& -V\l(\frac{\alpha\beta}{2V}\|\mu^t- \mu^{t-1}\|^2 +  \frac{V}{2\alpha\beta}\|\nabla f^{t-1}(\mu^{t-1})\|_*^2\r)
+ \frac{\alpha\beta}{2} \|\mu^t- \mu^{t-1}\|^2 \\
=& - \frac{V^{2}}{2\alpha\beta}\|\nabla f^{t-1}(\mu^{t-1})\|_{*}^2
\geq - \frac{V^{2}}{2\alpha\beta}D_1^2.
\end{align*}
where the first inequality follows from the strong convexity \eqref{eq:strong-convexity}, the second inequality follows from Holder's inequality, the third inequality follows from the fact that $ab\leq\frac{a^2+b^2}{2},~\forall a,b$, and the last inequality follows from the bound $\|\nabla f^{t-1}(\mu^{t-1})\|_*\leq D_1$. 
\end{proof}

\subsubsection{Proof of Lemma \ref{lem:constraint}}\label{sec:pf-constraint-1}
We start with a supporting lemma:
\begin{lemma}\label{lem:supp-constraint}
The updating rule \eqref{eq:Q-update} and \eqref{eq:H-update} delivers the following constraint violation bounds:
\begin{align*}
&\mathbb{E}\l\|\l[\frac1T\sum_{t=0}^{T-1}\overline{\mathbf{g}}(\mu^t)\r]_+\r\|_2
\leq \frac{\expect{\|\mf Q(t)\|_2}}{T} + \frac{D_2}{T}\sum_{t=0}^{T-1}\expect{\|\mu^{t+1} - \mu^t\|}\\
&\mathbb{E}\l\|\frac1T\sum_{t=0}^{T-1}\overline{\mf h}(\mu^t) - \mf b\r\|_2
\leq \frac{\expect{\|\mf H(t)\|_2}}{T} + \frac{H}{T}\sum_{t=0}^{T-1}\expect{\|\mu^{t+1} - \mu^t\|}
\end{align*}
\end{lemma}
\begin{proof}[Proof of Lemma \ref{lem:supp-constraint}]
We prove the first inequality and the second inequality is proved in the same way. Note that by \eqref{eq:Q-update}, we have
\begin{align*}
Q_i(t+1) =& \max \{ Q_i(t) + g_i^{t-1}(\mu^{t-1}) + \dotp{\nabla g_i^{t-1}(\mu^{t-1})}{\mu^t-\mu^{t-1}},~0\}\\
\geq& \max\{Q_i(t) + g_i^{t-1}(\mu^{t-1})- \|\nabla g_i^{t-1}(\mu^{t-1})\|_*\|\mu^t-\mu^{t-1}\|,~0\} \\
\geq& Q_i(t) + g_i^{t-1}(\mu^{t-1})- \|\nabla g_i^{t-1}(\mu^{t-1})\|_*\|\mu^t-\mu^{t-1}\|.
\end{align*}
Taking a telescoping sum from both sides from $0$ to $T-1$,
\[
Q_i(T) \geq \sum_{t=0}^{T-1}g_i^{t}(\mu^{t}) - \sum_{t=0}^{T-1}\|\nabla g_i^{t}(\mu^{t})\|_*\|\mu^{t+1}-\mu^{t}\|.
\]
This implies 
\[
\l[ \frac1T\sum_{t=0}^{T-1}g_i^{t}(\mu^{t})  \r]_+\leq \frac{Q_i(T)}{T} + \frac1T\sum_{t=0}^{T-1}\|\nabla g_i^{t}(\mu^{t-1})\|_*\|\mu^{t+1}-\mu^{t}\|,
\]
since the right hand side is nonnegative. Thus, we have
\begin{align*}
\l\|\l[\frac1T\sum_{t=0}^{T-1}\overline{\mathbf{g}}(\mu^t)\r]_+\r\|_2
\leq \frac{\|\mf Q(T)\|_2}{T}
+ \frac1T\sum_{t=0}^{T-1}\sqrt{\sum_{i=1}^L\|\nabla g_i^{t}(\mu^{t})\|_*^2}\|\mu^{t+1}-\mu^{t}\|\\
\leq \frac{\|\mf Q(T)\|_2}{T} + \frac{D_2}{T}\sum_{t=0}^{T-1}\|\mu^{t+1} - \mu^t\|,
\end{align*}
where the second inequality follows from Assumption \ref{as:basic}. 
\end{proof}

\begin{proof}[Proof of Lemma \ref{lem:constraint}]
It is enough to bound the difference $\expect{\|\mu^{t+1} - \mu^t\|}$. For this, we start from the relation \eqref{eq:interim-0} by taking $\mu=\mu^{t-1}$,
\begin{align*}
&V\dotp{\nabla f^{t-1}(\mu^{t-1})}{\mu^t}
+ \sum_{i=1}^{L} Q_i(t)\dotp{\nabla g_i^{t-1}(\mu^{t-1})}{\mu^t}
+  \sum_{j=1}^{M} H_i(t)\dotp{h_i^{t-1}}{\mu^t}  + \alpha D(\mu^t, \mu^{t-1})\nonumber\\
 \leq&  V\dotp{\nabla f^{t-1}(\mu^{t-1})}{\mu^{t-1}} + \sum_{i=1}^LQ_i(t)\dotp{\nabla g^{t-1}_{i}(\mu^{t-1})}{\mu^{t-1}} + \sum_{j=1}^{M} H_j(t)\dotp{h_i^{t-1}}{\mu^{t-1}}
  - \alpha D(\mu^{t-1}, \mu^t).
\end{align*}
By the fact that 
\begin{align*}
&\dotp{\nabla f^{t-1}(\mu^{t-1})}{\mu^{t-1}-\mu^{t}} 
\leq \|\nabla f^{t-1}(\mu^{t-1})\|_*\|\mu^{t-1}-\mu^{t}\|\leq D_1\|\mu^t-\mu^{t-1}\|\\
&\sum_{i=1}^LQ_i(t)\dotp{\nabla g^{t-1}_{i}(\mu^{t-1}) }{\mu^{t-1}-\mu^{t}}\\
&\leq \|\mf Q(t)\|_2\sqrt{\sum_{i=1}^L(\|\nabla g_i(\mu^{t-1})\|_*\|\mu^t-\mu^{t-1}\|)^2}\leq D_2\|\mf Q(t)\|_2
\|\mu^t-\mu^{t-1}\|, \\
&\sum_{j=1}^{M} H_j(t)\dotp{h_i^{t-1}}{\mu^{t-1}-\mu^{t}}
\leq \|\mf H(t)\|_2\sqrt{\sum_{i=1}^M(\|h_i^{t-1}\|_*\|\mu^t-\mu^{t-1}\|)^2}\leq  H\|\mf H(t)\|_2\|\mu^t-\mu^{t-1}\|,
\end{align*}
We get
\begin{align*}
D(\mu^t, \mu^{t-1}) + D(\mu^{t-1}, \mu^t)
\leq \frac{1}{\alpha}\l(VD_1 + D_2\|\mf Q(t)\|_2 + H\|\mf H(t)\|_2\r)\|\mu^t-\mu^{t-1}\|
\end{align*}
By strong convexity \eqref{eq:strong-convexity}, we have 
$$
D(\mu^t, \mu^{t-1}) + D(\mu^{t-1}, \mu^t)\geq\beta\|\mu^t - \mu^{t-1}\|^2
$$
Thus, it follows,
\begin{align*}
\beta\|\mu^t - \mu^{t-1}\|^2
\leq  \frac{1}{\alpha}\l(VD_1 + D_2\|\mf Q(t)\|_2 + H\|\mf H(t)\|_2\r)\|\mu^t-\mu^{t-1}\|.
\end{align*}
Solving the above quadratic inequality yields
\begin{align*}
\|\mu^t - \mu^{t-1}\|
\leq \frac{1}{\alpha\beta}\l(VD_1 + D_2\|\mf Q(t)\|_2 + H\|\mf H(t)\|_2\r).
\end{align*}
Taking the expectation from both sides and subtracting this bound into Lemma \ref{lem:supp-constraint} result in
\[
\mathbb{E}\l\|\l[\frac1T\sum_{t=0}^{T-1}\overline{\mathbf{g}}(\mu^t)\r]_+\r\|_2
\leq \frac{\expect{\|\mf Q(t)\|_2}}{T} + \frac{VD_1D_2}{\alpha\beta} + \frac{1}{T}\frac{D_{2}}{\alpha\beta} \sum_{t=1}^{T}\l(D_2\expect{\|\mf Q(t)\|_2} + H\expect{\|\mf H(t)\|_2}\r) 
\]
One can prove the bound on $\mathbb{E}\l\|\frac1T\sum_{t=0}^{T-1}\overline{\mf h}(\mu^t)\r\|_2$ with exactly the same computation and we omit the proof.
\end{proof}

\subsubsection{Proof of Lemma \ref{lem:dual-bound}}
For simplicity of notations, let constant $\overline{c}$ be the minimum over all $c_0$'s and let $\overline{\ell}$ be the maximum over all $\ell_0$'s in Lemma \ref{lem:leb} with $t=0,1,2,\cdots, T-1$ and $k=\sqrt{T}$.  We start with the following supporting lemma:
\begin{lemma}\label{lem:t0-drift}
Consider the $t_0$ slots drift for some positive integer $t_0$, then we have
\begin{align}\label{eq:interim-3}
&\frac{\|\mf Q(t+t_0)\|_2^2 + \|\mf H(t+t_0)\|_2^2 - \|\mf Q(t)\|_2^2 - \|\mf H(t)\|_2^2}{2} \nonumber \\
\leq &V \sum_{\tau=t}^{t+t_0-1}f^{\tau-1}(\mu) + 
\sum_{i=1}^LQ_i(t)\sum_{\tau=t}^{t+t_0-1}g^{\tau-1}_{i}(\mu) 
+ \sum_{j=1}^{M} H_j(t)\sum_{\tau=t}^{t+t_0-1}(\dotp{h_j^{\tau-1}}{\mu}-b_j)  +\frac{1}{2}C_{V,\alpha,t_0}.
\end{align}
\end{lemma}

\begin{proof}[Proof of Lemma \ref{lem:t0-drift}]
We start from equation \eqref{eq:interim-1}. Substituting \eqref{eq:res-bound}, we have 
\begin{multline*}
\Delta(t) + V(f^{t-1}(\mu^{t-1}) - f^{t-1}(\mu)) \leq 
\sum_{i=1}^LQ_i(t)g^{t-1}_{i}(\mu) + \sum_{j=1}^{M} H_j(t)(\dotp{h_j^{t-1}}{\mu}-b_j)\\
+ \frac{4RH^2}{\beta} + G^2+\frac{2RD_2^2}{\beta}+  \frac{V^{2}D_1^2}{2\alpha\beta}
+ \alpha D(\mu,\mu^{t-1}) - \alpha D(\mu, \mu^t).
\end{multline*} 
Take the summation from both sides between $t$ to $t+t_0-1$ for some $t_0$ to be determined later, we obtain
\begin{multline}\label{eq:drift-1}
\frac{\|\mf Q(t+t_0)\|_2^2 + \|\mf H(t+t_0)\|_2^2 - \|\mf Q(t)\|_2^2 - \|\mf H(t)\|_2^2}{2} \leq  \\
\sum_{\tau=t}^{t+t_0-1}\sum_{i=1}^LQ_i(\tau)g^{\tau-1}_{i}(\mu)  
+ \sum_{\tau=t}^{t+t_0-1}\sum_{j=1}^{M} H_j(\tau)(\dotp{h_j^{\tau-1}}{\mu}-b_j)
+V \sum_{\tau=t}^{t+t_0-1}(f^{\tau-1}(\mu^{\tau-1}) - f^{\tau-1}(\mu))\\
+\l( \frac{4RH^2}{\beta} + G^2+\frac{2RD_2^2}{\beta}+  \frac{V^{2}D_1^2}{2\alpha\beta} \r)t_0
+ \alpha D(\mu,\mu^{t-1}) - \alpha D(\mu, \mu^{t+t_0-1}).
\end{multline}

Using Assumption \ref{as:basic}, we have $V \sum_{\tau=t}^{t+t_0-1}f^{\tau-1}(\mu^{\tau-1})\leq VFt_{0}$.
Recall that $Q_i(t+1) = \max\{Q_i(t) + \tilde{g}_i^t,~0\}$, where  
$$\tilde{g}_i^t := g^{t-1}_{i}(\mu^{t-1}) + \dotp{\nabla g_{i}^{t-1} (\mu^{t-1})}{\mu^t-\mu^{t-1}},$$
and $H_j(t+1) = H_j(t) + \dotp{h_j^{t-1}}{\mu^t} - b_j$, we have

\begin{align}
&\sum_{\tau=t}^{t+t_0-1}\sum_{i=1}^L\l(Q_i(\tau)- Q_i(t)\r)g^{\tau-1}_{i}(\mu) \nonumber \\
\leq& \sum_{\tau=t+1}^{t+t_0-1}\sum_{i=1}^L (\sum_{\tau^{\prime}=t}^{\tau-1}|\tilde g_i^{\tau^{\prime}}| )\cdot |g^{\tau-1}_{i}(\mu)| \nonumber\\
\leq&\sum_{\tau=t+1}^{t+t_0-1}\sum_{\tau^{\prime}=t}^{\tau-1}\sum_{i=1}^L(|\tilde g_i^{\tau^{\prime}}|^2 + |g^{\tau-1}_{i}(\mu)|^2)/2  \nonumber\\
\leq&\sum_{\tau=t+1}^{t+t_0-1}\sum_{\tau^{\prime}=t}^{\tau-1}\sum_{i=1}^L|g^{\tau^{\prime}-1}_{i}(\mu^{\tau^{\prime}-1})|^2 + \|\nabla g^{\tau^{\prime}-1}_{i} \|_*^2\frac{2R}{\beta} 
+ |g^{\tau-1}_{i}(\mu)|^2/2  \nonumber\\
\leq& t_0^{2}\l(\frac32G^2 + \frac{2RD_2^2}{\beta}\r), \label{eq:inter-bound-g}
\end{align}
where the second from the last inequality follows from $\|\mu^t- \mu^{t-1}\|^2\leq2R/\beta$, and the
 last inequality follows from Assumption \ref{as:basic}. Similarly, we can show that
\begin{align}
\sum_{\tau=t}^{t+t_0-1}\sum_{j=1}^M\l(H_j(\tau)- H_j(t)\r)  (\dotp{h_j^{\tau-1}}{\mu^\tau} -b_j)
\leq t_0^{2}\frac{8RH^2}{\beta}.\label{eq:inter-bound-h}
\end{align}

Substituting the above two bounds into \eqref{eq:drift-1}, using the fact that $ \alpha D(\mu, \mu^{t+t_0-1})\geq0$ and $D(\mu,\mu^t)\leq R$, recalling that $C_{V,\alpha,t_0} := 2\l(\frac{4RH^2}{\beta} + G^{2}+\frac{2RD_2^2}{\beta}+  \frac{V^{2}D_1^2}{2\alpha\beta} + VF \r)t_0 + 2\l(\frac32G^2 + \frac{2RD_2^2}{\beta} + \frac{8RH^2}{\beta}\r)t_0^{2}+ 2\alpha R$
yields the desired result.

\end{proof}


\begin{proof}[Proof of Lemma \ref{lem:dual-bound}]
Thus, taking a conditional expectation from both sides conditioned on $\mathcal F^{t-1}$, we get
\begin{multline}\label{eq:interim-4}
\expect{\|\mf Q(t+t_0)\|_2^2 + \|\mf H(t+t_0)\|_2^2~|\mathcal F^{t-1}} - \|\mf Q(t)\|_2^2 - \|\mf H(t)\|_2^2\leq
2\expect{V \sum_{\tau=t}^{t+t_0-1}f^{\tau-1}(\mu)~|\mathcal{F}^{t-1}}\\
+2\sum_{i=1}^LQ_i(t)\expect{\sum_{\tau=t}^{t+t_0-1}g^{\tau-1}_{i}(\mu)}
+ 2\sum_{j=1}^{M} H_j(t)\expect{\sum_{\tau=t}^{t+t_0-1}(\dotp{h_j^{\tau-1}}{\mu}-b_j)}  + C_{V,\alpha,t_0},
\end{multline}
where we use the following two facts: (1) $\mf Q(t),~\mf H(t)\in\mathcal{F}^{t-1}$. (2) $g_i^\tau$ and $h_i^\tau$ are independent of system history $\mathcal{F}^{t-1}$ and thus the conditional expectation equals the expectation.

Note that by definition, $f^t(\mu) = f(\mu,\xi^t)$, and
according to the notation in \eqref{eq:partial-static-prob},
\[
\expect{V \sum_{\tau=t}^{t+t_0-1}f^{\tau-1}(\mu)~|\mathcal{F}^{t-1}}
= \expect{\l.V \mathbb{E}_{\xi}\l[\sum_{\tau=t}^{t+t_0-1}f^{\tau-1}(\mu)\r]~\r|\mathcal{F}^{t-1}}
=Vt_0\expect{\overline{f}^{(t,t_0)}(\mu) ~| \mathcal{F}^{t-1}}.
\]
Furthermore, 
\[
\expect{\sum_{\tau=t}^{t+t_0-1}g^{\tau-1}_{i}(\mu)} = t_0 \overline{g}_i(\mu),
~~~\expect{\sum_{\tau=t}^{t+t_0-1}\dotp{h_j^{\tau-1}}{\mu}} = t_0 \overline{h}_j(\mu).
\]
Substituting these three relations into \eqref{eq:interim-4}, we get
\begin{align}\label{eq:interim-5}
&\expect{\|\mf Q(t+t_0)\|_2^2 + \|\mf H(t+t_0)\|_2^2~|\mathcal F^{t-1}} - \|\mf Q(t)\|_2^2 - \|\mf H(t)\|_2^2 \nonumber\\
\leq& 2Vt_0\expect{ \overline{f}^{(t,t_0)}(\mu)~|\mathcal{F}^{t-1}}
+2t_0\l(\sum_{i=1}^LQ_i(t)\overline{g}_i(\mu)
+ 2\sum_{j=1}^{M} H_j(t)(\overline{h}_j(\mu)-b_j)\r) + C_{V,\alpha,t_0} \nonumber\\
\leq& 2Vt_0\expect{ \overline{f}^{(t,t_0)}(\mu) + 
\sum_{i=1}^L\frac{Q_i(t)}{V}\overline{g}_i(\mu)
+ \sum_{j=1}^{M} \frac{H_j(t)}{V}(\overline{h}_j(\mu) - b_j)~|\mathcal{F}^{t-1}} + C_{V,\alpha,t_0}.
\end{align}
The main idea here, as is mentioned in the proof outline, is to realize that
\[
q^{(t,t_0)}\big(\frac{\mf Q(t)}{V},~\frac{\mf H(t)}{V}\big)=\min_{\mu\in\Delta}~~\overline{f}^{(t,t_0)}(\mu) + 
\sum_{i=1}^L\frac{Q_i(t)}{V}\overline{g}_i(\mu)
+ \sum_{j=1}^{M} \frac{H_j(t)}{V}(\overline{h}_j(\mu)-b_j),
\]
where $q^{(t,t_0)}$ is the Lagrangian dual function defined in \eqref{eq:target-dual} with dual variables 
$[\frac{\mf Q(t)}{V},~\frac{\mf H(t)}{V}]$. This implies if we choose $\mu = \mu_0$ in \eqref{eq:interim-5} as one of the solutions to the above problem, then, we can transform the bound \eqref{eq:interim-5} to \eqref{eq:drift-bound-1} and we finish the proof.
\end{proof}


\subsubsection{Proof of Lemma \ref{lem:bound-dual-function}}
\begin{proof}[Proof of Lemma \ref{lem:bound-dual-function}]
Now, we take $t_0=\sqrt{T}$ and by SELM (Assumption \ref{as:selm}), there exists a solution to the maximization problem 
$$
\Lambda^*:=\text{argmax}_{\lambda,\eta}~~q^{(t,t_0)}(\lambda,\eta).
$$
Let $(\lambda^*,~\eta^*)$ be one of the solutions to this problem. Recall that we define $\overline{c}$ to be the minimum over all $c_0$'s and define $\overline{\ell}$ to be the maximum over all $\ell_0$'s in Lemma \ref{lem:leb} with $t=0,1,2,\cdots, T-1$ and $k=\sqrt{T}$. If 
$\text{dist}\l(\big(\frac{\mf Q(t)}{V},~\frac{\mf H(t)}{V}\big), ~\Lambda^*\r)\geq\overline{\ell}$, then, by Lemma \ref{lem:leb} we have
\begin{align*}
&q^{(t,t_0)}\big(\frac{\mf Q(t)}{V},~\frac{\mf H(t)}{V}\big) \\ 
=& q^{(t,t_0)}\big(\frac{\mf Q(t)}{V},~\frac{\mf H(t)}{V}\big)  - q^{(t,t_0)}(\lambda^*,~\eta^*) +  q^{(t,t_0)}(\lambda^*,~\eta^*)\\
\leq&-\overline{c}\cdot\text{dist}\l(\big(\frac{\mf Q(t)}{V},~\frac{\mf H(t)}{V}\big), ~\Lambda^*\r)
+ q^{(t,t_0)}(\lambda^*,~\eta^*)\\
\leq&
-\overline{c}\cdot\text{dist}\l(\big(\frac{\mf Q(t)}{V},~\frac{\mf H(t)}{V}\big), ~\Lambda^*\r) + \overline{f}^{(t,t_0)}(\mu_0)\\
\leq& -\overline{c}\Big\|\big(\frac{\mf Q(t)}{V},~\frac{\mf H(t)}{V}\big)\Big\|_2 + \overline{c}B + F,
\end{align*}
where the first inequality follows from Lemma \ref{lem:leb}, the second inequality follows from choosing $\mu_0$ as the solution to the following problem 
\[
\min_{\mu\in\Delta}\overline{f}^{t,t_0}(\mu)~~s.t.~~\overline{\mathbf{g}}(\mu)\leq0,~~\overline{\mathbf{h}}(\mu) =\mf b,
\]
and using weak duality (in fact, by Lemma \ref{cor:mfcq}, we know KKT conditions hold for this problem and strong duality holds). The third inequality follows from triangle inequality and the boundedness of Lagrange multipliers $\max_{[\lambda,\mu]\in\mathcal{V}^*}\|[\lambda,\mu]\|_2\leq B$. 

On the other hand, if 
$\text{dist}\l(\big(\frac{\mf Q(t)}{V},~\frac{\mf H(t)}{V}\big), ~\Lambda^*\r)<\overline{\ell}$, then, one has
\begin{align*}
&q^{(t,t_0)}\big(\frac{\mf Q(t)}{V},~\frac{\mf H(t)}{V}\big)\\
=& 
\min_{\mu\in\Delta}~~\overline{f}^{(t,t_0)}(\mu) + 
\sum_{i=1}^L\frac{Q_i(t)}{V}\overline{g}_i(\mu)
+ \sum_{j=1}^{M} \frac{H_j(t)}{V}(\overline{h}_j(\mu)-b_j)\\
=&
\min_{\mu\in\Delta}~~\overline{f}^{(t,t_0)}(\mu) + 
\sum_{i=1}^L\Big(\lambda_i^*\overline{g}_i(\mu) + \dotp{\frac{Q_i(t)}{V} - \lambda_i^*}{\overline{g}_i(\mu)} \Big)
+ \sum_{j=1}^{M} \Big(\mu_j^*\overline{h}_j(\mu)  - b_j  +\dotp{\frac{H_j(t)}{V} - \mu_j^*}{\overline{h}_j(\mu) }\Big)\\
\leq& q^{(t,t_0)}(\lambda^*,~\eta^*)  + \overline\ell\Big(G+\sqrt{\frac{2RH^2}{\beta}}\Big)
\leq F  + \overline\ell\Big(G+\sqrt{\frac{2RH^2}{\beta}}\Big),
\end{align*}
where we choose $(\lambda^*,\mu^*)$ to be a point in $\Lambda^*$ closest to $\big(\frac{\mf Q(t)}{V},~\frac{\mf H(t)}{V}\big)$, 
the first inequality follows from 
\begin{align*}
& \sum_{i=1}^L\dotp{\frac{Q_i(t)}{V} - \lambda_i^*}{\overline{g}_i(\mu)}
 \leq \|\mf Q(t)/V - \lambda^*\|_2\|\overline{\mf g}(\mu)\|_2\leq G\overline{\ell}\\
 &\sum_{j=1}^M \dotp{\frac{H_j(t)}{V} - \mu_j^*}{\overline{h}_j(\mu) }
 \leq \|\mf H(t)/V - \mu^*\|_2\|\overline{\mf h}(\mu)\|_2\leq \sqrt{\frac{2RH^2}{\beta}}\overline{\ell},
\end{align*}
and the second inequality follows from weak duality. Overall, we finish the proof.
\end{proof}

\subsubsection{Proof of Lemma \ref{lem:q-bound}}\label{sec:pf-constraint-2}
Substituting Lemma \ref{lem:bound-dual-function} into \eqref{eq:drift-bound-1}, we get

\begin{align}\label{eq:drift-bound-2}
&\expect{\|\mf Q(t+t_0)\|_2^2 + \|\mf H(t+t_0)\|_2^2~|\mathcal F^{t-1}} - \|\mf Q(t)\|_2^2 - \|\mf H(t)\|_2^2 \nonumber \\
\leq &  C_{V,\alpha,t_0}+
2\Big(F + \overline{\ell}(G+\sqrt{2RH^2/\beta} +\overline{c}) + \overline{c}B \Big)Vt_0 - 2\overline{c}t_0\Big\|\big(\mf Q(t),~\mf H(t)\big)\Big\|_2.
\end{align}
This bound is the key to our analysis. Intuitively, it says if $\|\big(\mf Q(t),~\mf H(t)\big)\|_2$ is very large at certain time slot $t$, then, $\|\big(\mf Q(t+t_0),~\mf H(t+t_0)\big)\|_2$ becomes very small. Since $\|\big(\mf Q(t+t_0),~\mf H(t+t_0)\big)\|_2$ is nonnegative, this means $\|\big(\mf Q(t),~\mf H(t)\big)\|_2$ cannot be too large to start with. 
To transform this intuition into a uniform bound on $\big(\mf Q(t),~\mf H(t)\big)$ over all time slots, we invoke the following drift lemma:

\begin{lemma}[Lemma 5 of \citep{yu2017online}] \label{lm:drift-random-process-bound}
Let $\{Z(t), t\geq 1\}$ be a discrete time stochastic process adapted to a filtration $\{\mathcal{F}(t), t\geq 1\}$ with $Z(0) = 0$ and $\mathcal{F}(0) = \{\emptyset, \Omega\}$.  Suppose there exist integer $t_{0}>0$, real constants $\theta\in \mathbb{R}$, $\delta_{\max} > 0$ and $0< \zeta \leq \delta_{\max}$ such that 
\begin{align} 
\vert Z(t+1) - Z(t) \vert \leq& \delta_{\max}, \label{eq:bounded-difference}\\
\mathbb{E}[ Z(t+t_{0}) - Z(t) | \mathcal{F}(t)] \leq & \left\{ \begin{array}{cc} t_{0}\delta_{\max}, &\text{if}~ Z(t) < \theta \\  -t_{0}\zeta , &\text{if}~ Z(t) \geq \theta \end{array}\right.. \label{eq:stochastic-process-drift-condition}
\end{align}
hold for all $t\in \{1,2,\ldots\}$. Then, $\mathbb{E}[Z(t)] \leq  \theta + t_{0}\frac{4\delta_{\max}^{2}}{\zeta} \log\big[\frac{8\delta_{\max}^{2}}{\zeta^{2}}\big], \forall t\in\{1,2,\ldots\}.$
\end{lemma}

To apply this lemma,
we set $Z(t) = \| \big(\mf Q(t),~\mf H(t)\big) \|_2$ and we need to check condition \eqref{eq:bounded-difference} and \eqref{eq:stochastic-process-drift-condition}, for which we prove the following lemma:

\begin{proof}[Proof of Lemma \ref{lem:q-bound}]
For condition \eqref{eq:bounded-difference}, we have
\begin{align*}
&\l \vert \| \big(\mf Q(t+1),~\mf H(t+1)\big) \|_2 -  \| \big(\mf Q(t),~\mf H(t)\big) \|_2  \r \vert \\
 \leq &\| \big(\mf Q(t+1) - \mf Q(t),~\mf H(t+1) - \mf H(t)\big) \|_2
 \leq \sqrt{\sum_{i=1}^L(\tilde g_i^t)^2} + \sqrt{\sum_{j=1}^M(\dotp{h_j^t}{\mu^t}-b_j)^2}\\
 \leq &2\Big(G+\sqrt{\frac{2RD_2^2}{\beta}}\Big) + \sqrt{\frac{8RH^2}{\beta}}.  
\end{align*}
On the other hand, for condition \eqref{eq:stochastic-process-drift-condition} we start from \eqref{eq:drift-bound-2}. 
Suppose 
\[
\Big\|\big(\mf Q(t),~\mf H(t)\big)\Big\|_2\geq
\frac{C_{V,\alpha,t_0}+
2\Big(F + \overline{\ell}(G+\sqrt{2RH^2/\beta} +\overline{c}) + \overline{c}B \Big)Vt_0}{\overline c t_0},
\]
then,  \eqref{eq:drift-bound-2} can be rewritten as
\[
\expect{\|\mf Q(t+t_0)\|_2^2 + \|\mf H(t+t_0)\|_2^2~|\mathcal F^{t-1}} - \|\mf Q(t)\|_2^2 - \|\mf H(t)\|_2^2\\
\leq  - \overline c t_0\|\big(\mf Q(t),~\mf H(t)\big)\Big\|_2 + \overline{c}^2t_0^2,
\]
which implies
\[
\expect{\|\mf Q(t+t_0)\|_2^2 + \|\mf H(t+t_0)\|_2^2~|\mathcal F^{t-1}}
\leq \l(\|\big(\mf Q(t),~\mf H(t)\big)\Big\|_2 -  \overline c t_0 \r)^2,
\]
Taking square root from both sides and by Jensen's inequality, we have
\[
\expect{\big\|\big(\mf Q(t+t_{0}),~\mf H(t+t_{0})\big)\Big\|_2~|\mathcal F^{t-1}}\leq \big\|\big(\mf Q(t),~\mf H(t)\big)\Big\|_2 
-  \overline c t_0.
\]

Overall, by Lemma \ref{lm:drift-random-process-bound}, we obtain
\begin{multline*}
\expect{\Big\|\big(\mf Q(t),~\mf H(t)\big)\Big\|_2}
\leq 
\frac{C_{V,\alpha,t_0}+
2\Big(F + \overline{\ell}(G+\sqrt{8RH^2/\beta} +\overline{c}) + \overline{c}B \Big)Vt_0}{\overline c t_0}\\
+ \frac{4t_0\Big(2 (G+\sqrt{2RD_2^2/\beta}) 
+ \sqrt{8RH^2/\beta}\Big)^2}{\overline c}\log\l(\frac{8\l(2(G+\sqrt{\frac{2RD_2^2}{\beta}}) + \sqrt{\frac{8RH^2}{\beta}}\r)^2}{\overline c^2}\r),
\end{multline*}
Taking $V=\sqrt{T}, \alpha = T$ and $t_{0} = \sqrt{T}$ and recalling the definition of $C_{V, \alpha, t_{0}}$ yields:
 \begin{align*}
\expect{\Big\|\big(\mf Q(t),~\mf H(t)\big)\Big\|_2} \leq C^{\prime} + C^{\prime\prime} \sqrt{T}
\end{align*}
where  $C^{\prime} := \frac{2}{\overline{c}}\l(\frac{4RH^2}{\beta} + G^{2}+\frac{2RD_2^2}{\beta}+  \frac{D_1^2}{2\beta}\r)$ and $C^{\prime\prime}:= \frac{2}{\overline{c}}\Big(2F + \frac{3}{2}G^2 + \frac{2RD_2^2}{\beta} + \frac{8RH^2}{\beta} + R + \overline{\ell}(G+\sqrt{8RH^2/\beta} +\overline{c}) + \overline{c}B + 2\big(2 (G+\sqrt{2RD_2^2/\beta}) + \sqrt{8RH^2/\beta}\big)^2 \log\big(\frac{8 (2(G+\sqrt{\frac{2RD_2^2}{\beta}}) + \sqrt{\frac{8RH^2}{\beta}})^2}{\overline c^2}\big)\Big)$ are absolute constants.
\end{proof}


\subsection{Proof of Theorem \ref{main:theorem-simplex}}\label{sec:pf-prob-simplex}
In this section, we present the proof for Theorem \ref{main:theorem-simplex}. The proof takes into account the fact that $\Delta$ is the probability simplex and the effect of pull-away operation $\tilde{\mu}^{t-1} = (1-\theta)\mu^{t-1} + \frac{\theta}{d}\mathbf{1}$. 
Note that in this probability simplex case, we have $\sup_{\mu_1,\mu_2\in\Delta}\|\mu_1- \mu_2\|_{1}\leq 1$, which will be used to replace the frequently used relation $\sup_{\mu_1,\mu_2\in\Delta}\|\mu_1- \mu_2\|\leq \sqrt{\frac{2R}{\beta}}$ in the proof for general cases. Note further that when $\Delta$ is the probability simplex and $D(\mu_{1}, \mu_{2})$ is chosen to be K-L divergence, we do not have a uniform bound $R$ such that $\sup_{\mu_1,\mu_2\in\Delta} D(\mu_{1}, \mu_{2})\leq R$. Fortunately, our analysis does not need such a uniform bound but instead uses a bound on $D(\mu_1,\tilde{\mu}_2)$ where $\tilde{\mu}_{2}$ is in the form of $\tilde{\mu}^{t}$ specified in Algorithm \ref{alg:prob-simplex}.

The following lemma bounds the difference between $D(\mu,\tilde{\mu}^{t-1}) $ and $D(\mu, \mu^{t-1}) $:
\begin{lemma}\label{lem:diff}
Consider any $\mu_1,\mu_2\in\Delta$, and let $\tilde{\mu}_2 = (1-\theta)\mu_2 + \theta\frac{1}{d}\mathbf{1}$, for some 
$\theta\in(0,1]$, then, it follows
\[
D(\mu_1,\tilde{\mu}_2) - D(\mu_1,\mu_2) \leq \theta\log d.
\]
Furthermore, 
$
D(\mu_1,\tilde{\mu}_2)\leq \log(d/\theta).
$
\end{lemma}
\begin{proof}[Proof of Lemma \ref{lem:diff}]
\begin{align*}
&D(\mu_1,\tilde{\mu}_2) - D(\mu_1,\mu_2)\\
=&\sum_{i=1}^d\mu_1(i)\l( \log\frac{\mu_1(i)}{\tilde\mu_2(i)} - \log\frac{\mu_1(i)}{\mu_2(i)} \r)\\
=&\sum_{i=1}^d\mu_1(i)\log\frac{\mu_2(i)}{\tilde\mu_2(i)}\\
=&\sum_{i=1}^d\mu_1(i)\l(\log\mu_2(i) - \log\l((1-\theta)\mu_2(i) + \theta\frac{1}{d}\mathbf{1}\r)\r)\\
\leq&\sum_{i=1}^d\mu_1(i)\l(\log\mu_2(i) - (1-\theta)\log\mu_2(i) - \theta\log \frac{1}{d}\r)\\
=&\theta\sum_{i=1}^d\mu_1(i)\l(\log\mu_2(i)  + \log d\r) \leq \theta\log d,
\end{align*}
where the first inequality follows from the concavity of $\log$ function. Furthermore, for the second inequality, we have
\begin{multline*}
D(\mu_1,\tilde{\mu}_2)=\sum_{i=1}^d\mu(i)\log\frac{\mu_1(i)}{\tilde{\mu}_2(i)}
=\sum_{i=1}^d\mu(i)\log\frac{\mu_1(i)}{(1-\theta)\mu_2^{t-1}(i) + \theta/d}\\
\leq -\sum_{i=1}^d\mu_1(i)\log((1-\theta)\mu_2^{t-1}(i) + \theta/d)\leq \log(d/\theta).
\end{multline*}
\end{proof}

\subsubsection{Regret bound}
First of all, by the same proof as that of Lemma \ref{lem:strong-convex-queue} one can show the following:
\begin{align}
&V\dotp{\nabla f^{t-1}(\mu^{t-1})}{\mu^t-\mu^{t-1}}+\Delta(t) + \alpha D(\mu^t, \tilde\mu^{t-1})\nonumber\\
 \leq&  V(f^{t-1}(\mu) - f^{t-1}(\mu^{t-1})) + \sum_{i=1}^LQ_i(t)g^{t-1}_{i}(\mu) + \sum_{j=1}^{M} H_j(t)(\dotp{h_j^{t-1}}{\mu}-b_j)
 \nonumber\\
 &+ \alpha D(\mu,\tilde\mu^{t-1}) - \alpha D(\mu, \mu^t)
 + H^2 + G^2+ D_2^2.  \label{eq:interim-11}
\end{align}
Furthermore, similar to that of Lemma \ref{lem:lhs-lower-bound}, we have
\begin{align}
&V\dotp{\nabla f^{t-1}(\mu^{t-1})}{\mu^t-\mu^{t-1}} + \alpha D(\mu^t, \tilde{\mu}^{t-1}) \nonumber\\
\geq& V\dotp{\nabla f^{t-1}(\mu^{t-1})}{\mu^t-\tilde{\mu}^{t-1}} + \alpha D(\mu^t, \tilde{\mu}^{t-1})
- V\theta\|\nabla f^{t-1}(\mu^{t-1})\|_{\infty}  \nonumber\\
\geq&V\dotp{\nabla f^{t-1}(\mu^{t-1})}{\mu^t-\tilde{\mu}^{t-1}} + \frac{\alpha}{2} \|\mu^t- \tilde{\mu}^{t-1}\|_1^2
- V\theta\|\nabla f^{t-1}(\mu^{t-1})\|_{\infty} \nonumber\\
\geq&-V \|\nabla f^{t-1}(\mu^{t-1})\|_{\infty} \|\mu^t-\tilde{\mu}^{t-1}\|_1 + \frac{\alpha}{2} \|\mu^t- \tilde{\mu}^{t-1}\|_1^2
- V\theta\|\nabla f^{t-1}(\mu^{t-1})\|_{\infty} \nonumber\\
\geq& -V\l(\frac{\alpha}{2V}\|\mu^t- \tilde{\mu}^{t-1}\|_1^2 +  \frac{V}{2\alpha}\|\nabla f^{t-1}(\mu^{t-1})\|_{\infty}^2\r)
+ \frac{\alpha}{2} \|\mu^t- \tilde{\mu}^{t-1}\|_1^2 - V\theta\|\nabla f^{t-1}(\mu^{t-1})\|_{\infty} \nonumber\\
=& - (\frac{V^{2}}{2\alpha}\|\nabla f^{t-1}(\mu^{t-1})\|_{\infty}^2 +V\theta\|\nabla f^{t-1}(\mu^{t-1})\|_{\infty})
\geq - \frac{VD_1^2}{2\alpha} - V\theta D_1.
\label{eq:res-bound-1}
\end{align}
Substituting \eqref{eq:res-bound-1} into \eqref{eq:interim-11} gives
\begin{multline}\label{eq:interim-21}
\Delta(t) + V(f^{t-1}(\mu^{t-1}) - f^{t-1}(\mu)) \leq 
\sum_{i=1}^LQ_i(t)g^{t-1}_{i}(\mu) + \sum_{j=1}^{M} H_j(t)(\dotp{h_j^{t-1}}{\mu}-b_j)\\
+H^2 + G^2+D_2^2 +  \frac{V^{2}}{2\alpha}D_1^2 +V\theta D_1 
+ \alpha D(\mu,\tilde\mu^{t-1}) - \alpha D(\mu, \mu^t).
\end{multline}
Using Lemma \ref{lem:diff}, we get
\begin{multline*}
\Delta(t) + V(f^{t-1}(\mu^{t-1}) - f^{t-1}(\mu)) \leq 
\sum_{i=1}^LQ_i(t)g^{t-1}_{i}(\mu) + \sum_{j=1}^{M} H_j(t)(\dotp{h_j^{t-1}}{\mu} - b_j)\\
+H^2 + G^2+D_2^2 +  \frac{V^{2}}{2\alpha}D_1^2 +V\theta D_1 + \alpha\theta\log d
+ \alpha D(\mu,\mu^{t-1}) - \alpha D(\mu, \mu^t).
\end{multline*}
The rest follows from the same argument as that of Section \ref{sec:reg-analysis} after \eqref{eq:roadmap-2} and we omit the details for brevity.

\subsubsection{Constraint violations}
Similar as before, we start with the following lemma:
\begin{lemma}\label{lem:simplex-constraint}
The updating rule \eqref{eq:Q-update} and \eqref{eq:H-update} delivers the following constraint violation bounds:
\begin{align*}
&\mathbb{E}\l\|\l[\frac1T\sum_{t=0}^{T-1}\overline{\mathbf{g}}(\mu^t)\r]_+\r\|_2
\leq \frac{\expect{\|\mf Q(t)\|_2}}{T} + \frac{2D_2}{\alpha}\l(VD_1 + D_2\expect{\|\mf Q(t)\|_2} + H\expect{\|\mf H(t)\|_2}\r) 
+ D_2\theta,\\
&\mathbb{E}\l\|\frac1T\sum_{t=0}^{T-1}\overline{\mf h}(\mu^t) - \mf b\r\|_2
\leq \frac{\expect{\|\mf H(t)\|_2}}{T} +  \frac{2H}{\alpha}\l(VD_1 + D_2\expect{\|\mf Q(t)\|_2} + H\expect{\|\mf H(t)\|_2}\r)
+H\theta .
\end{align*}
\end{lemma}
\begin{proof}[Proof of Lemma \ref{lem:simplex-constraint}]
Using Lemma \ref{lem:supp-constraint}, it is enough to bound the difference $\expect{\|\mu^{t+1} - \mu^t\|_1}$. For this, 
applying Lemma \ref{lem:strong-convex} by setting $y = \mu^{t-1}$, $x^* = \mu^t$, and  $f(x) =\dotp{x}{p}$ with
$$
p = V\nabla f^{t-1}(\mu^{t-1}) + \sum_{i=1}^{L} Q_i(t)\nabla g_i^{t-1}(\mu^{t-1}) 
 + \sum_{j=1}^{M} H_j(t)h_j^{t-1},
$$ 
we have
\begin{multline}\label{eq:interim-00}
 \dotp{V\nabla f^{t-1}(\mu^{t-1}) + \sum_{i=1}^{L} Q_i(t)\nabla g_i^{t-1}(\mu^{t-1}) 
 + \sum_{j=1}^{M} H_i(t)h_j^{t-1}}{\mu^t} + \alpha D(\mu_t,\tilde{\mu}^{t-1}) \\
 \leq
  \dotp{V\nabla f^{t-1}(\mu^{t-1}) + \sum_{i=1}^{L} Q_i(t)\nabla g_i^{t-1}(\mu^{t-1}) 
 + \sum_{j=1}^{M} H_j(t)h_j^{t-1}}{\mu} + \alpha (D(\mu,\tilde{\mu}^{t-1}) - D(\mu,\mu^t)).
\end{multline}
Taking $\mu=\tilde\mu^{t-1}$ in  \eqref{eq:interim-00} gives,
\begin{align*}
&V\dotp{\nabla f^{t-1}(\mu^{t-1})}{\mu^t}
+ \sum_{i=1}^{L} Q_i(t)\dotp{\nabla g_i^{t-1}(\mu^{t-1})}{\mu^t}
+  \sum_{j=1}^{M} H_j(t)\dotp{h_j^{t-1}}{\mu^t}  + \alpha D(\mu^t, \tilde{\mu}^{t-1})\nonumber\\
 \leq&  V\dotp{\nabla f^{t-1}(\mu^{t-1})}{\tilde\mu^{t-1}} + \sum_{i=1}^LQ_i(t)\dotp{\nabla g^{t-1}_{i}(\mu^{t-1})}{\tilde\mu^{t-1}} + \sum_{j=1}^{M} H_j(t)\dotp{h_j^{t-1}}{\tilde\mu^{t-1}}\\
 & - \alpha D(\tilde\mu^{t-1}, \mu^t).
\end{align*}
By the fact that 
\begin{align*}
&\dotp{\nabla f^{t-1}(\mu^{t-1})}{\tilde\mu^{t-1}-\mu^t} 
\leq \|\nabla f^{t-1}(\mu^{t-1})\|_\infty\|\mu^t-\tilde\mu^{t-1}\|_1\leq D_1\|\mu^t-\tilde\mu^{t-1}\|_1\\
&\sum_{i=1}^LQ_i(t)\dotp{\nabla g^{t-1}_{i}(\mu^{t-1}) }{\tilde\mu^{t-1}-\mu^t}\\
&\leq \|\mf Q(t)\|_2\sqrt{\sum_{i=1}^L(\|\nabla g_i(\mu^{t-1})\|_\infty\|\mu^t-\tilde\mu^{t-1}\|_1)^2}\leq D_2\|\mf Q(t)\|_2
\|\mu^t-\tilde\mu^{t-1}\|_1, \\
&\sum_{j=1}^{M} H_j(t)\dotp{h_j^{t-1}}{\tilde\mu^{t-1}-\mu^t}
\leq \|\mf H(t)\|_2\sqrt{\sum_{j=1}^M(\|h_j^{t-1}\|_{\infty}\|\mu^t-\tilde\mu^{t-1}\|_1)^2}\leq  H\|\mf H(t)\|_2\|\mu^t-\tilde\mu^{t-1}\|_1,
\end{align*}
We get
\begin{align*}
D(\mu^t, \tilde{\mu}^{t-1}) + D(\tilde\mu^{t-1}, \mu^t)
\leq \frac{1}{\alpha}\l(VD_1 + D_2\|\mf Q(t)\|_2 + H\|\mf H(t)\|_2\r)\|\mu^t-\tilde\mu^{t-1}\|_1.
\end{align*}
By Pinsker's inequality, we have 
$$
D(\mu^t, \tilde{\mu}^{t-1}) + D(\tilde\mu^{t-1}, \mu^t)
\geq \|\mu^t - \mu^{t-1}\|_1^2 
$$
Thus, it follows,
\begin{align*}
\|\mu^t - \tilde\mu^{t-1}\|_1^2
\leq 2\theta^{2}+ \frac{1}{\alpha}\l(VD_1 + D_2\|\mf Q(t)\|_2 + H\|\mf H(t)\|_2\r)\|\mu^t-\tilde\mu^{t-1}\|_1.
\end{align*}
Solving the above quadratic inequality
\begin{align*}
\|\mu^t - \tilde\mu^{t-1}\|_1
\leq \frac{2}{\alpha}\l(VD_1 + D_2\|\mf Q(t)\|_2 + H\|\mf H(t)\|_2\r) + 2\theta,
\end{align*}
which implies
\begin{align*}
\|\mu^t - \mu^{t-1}\|_1
\leq \frac{2}{\alpha}\l(VD_1 + D_2\|\mf Q(t)\|_2 + H\|\mf H(t)\|_2\r) + 3\theta,
\end{align*}

Taking the expectation from both sides and subtracting this bound into Lemma \ref{lem:supp-constraint} results in
\[
\mathbb{E}\l\|\l[\frac1T\sum_{t=0}^{T-1}\overline{\mathbf{g}}(\mu^t)\r]_+\r\|_2
\leq \frac{\expect{\|\mf Q(t)\|_2}}{T} + 3\theta D_{2} + \frac{2VD_{1}D_{2}}{\alpha}+\frac{1}{T} \sum_{t=0}^{T-1}\frac{2D_2}{\alpha}\l(D_2\expect{\|\mf Q(t)\|_2} + H\expect{\|\mf H(t)\|_2}\r)
\]
One can prove the bound on $\mathbb{E}\l\|\frac1T\sum_{t=0}^{T-1}\overline{\mf h}(\mu^t) - \mf b\r\|_2$ with exactly the same computation and we omit the proof.
\end{proof}

Now, by Lemma \ref{lem:simplex-constraint} it is enough to bound $\mf Q(t)$ and $\mf H(t)$, for which we have the following lemma:

\begin{lemma}\label{lem:t0-drift-0}
Consider the $t_0$ slots drift for some positive integer $t_0$, then we have
\begin{multline}\label{eq:interim-30}
\frac{\|\mf Q(t+t_0)\|_2^2 + \|\mf H(t+t_0)\|_2^2 - \|\mf Q(t)\|_2^2 - \|\mf H(t)\|_2^2}{2}\leq
V \sum_{\tau=t}^{t+t_0-1}f^{\tau-1}(\mu)\\
\sum_{i=1}^LQ_i(t)\sum_{\tau=t}^{t+t_0-1}g^{\tau-1}_{i}(\mu) 
+ \sum_{j=1}^{M} H_j(t)\sum_{\tau=t}^{t+t_0-1}(\dotp{h_j^{\tau-1}}{\mu} - b_j) +\frac{1}{2}\hat{C}_{V,\alpha,t_{0}},
\end{multline}
where $$\hat{C}_{V,\alpha,t_0}:= 2\big(H^{2}+\frac{3}{2}G^{2}+D_{2}^{2}\big)t_{0}^{2}
 +2\big(H^2 + G^2+D_2^2 +  \frac{V^{2}}{2\alpha}D_1^2 +V\theta D_1+\alpha\theta\log d\big)t_0
+ 2\alpha \log(d/\theta)$$
\end{lemma}
\begin{proof}[Proof of Lemma \ref{lem:t0-drift-0}]
First of all, summing both sides of \eqref{eq:interim-21} from $\tau=t$ to $\tau = t+t_0-1$ gives
\begin{multline}\label{eq:drift-10}
\frac{\|\mf Q(t+t_0)\|_2^2 + \|\mf H(t+t_0)\|_2^2 - \|\mf Q(t)\|_2^2 - \|\mf H(t)\|_2^2}{2}  \\
\leq \sum_{\tau=t}^{t+t_0-1}\sum_{i=1}^LQ_i(\tau)g^{\tau-1}_{i}(\mu) 
+ \sum_{\tau=t}^{t+t_0-1}\sum_{j=1}^{M} H_j(\tau)(\dotp{h_j^{\tau-1}}{\mu}-b_j)
+V \sum_{\tau=t}^{t+t_0-1}(f^{\tau-1}(\mu^{\tau-1}) - f^{\tau-1}(\mu))\\
+\l(H^2 + G^2+2D_2^2 +  \frac{V^{2}}{2\alpha}D_1^2 +V\theta D_1\r)t_0
+ \alpha D(\mu,\tilde{\mu}^{t-1}) - \alpha D(\mu, \mu^{t+t_0-1})\\
+\alpha\sum_{\tau =t+1}^{t+t_0-1}(D(\mu,\tilde{\mu}^{\tau-1}) - D(\mu,\mu^{\tau-1})).
\end{multline}
By Lemma \ref{lem:diff}, one has
\[
\alpha\sum_{\tau =t+1}^{t+t_0-1}(D(\mu,\tilde{\mu}^{\tau-1}) - D(\mu,\mu^{\tau-1}))\leq t_0\alpha\theta\log d.
\]
and 
\[
\alpha D(\mu,\tilde{\mu}^{t-1})\leq \alpha\log(d/\theta),
\]
Thus, substituting these two bounds into \eqref{eq:drift-10} gives
\begin{multline}\label{eq:drift-20}
\frac{\|\mf Q(t+t_0)\|_2^2 + \|\mf H(t+t_0)\|_2^2 - \|\mf Q(t)\|_2^2 - \|\mf H(t)\|_2^2}{2} \leq \\
\sum_{\tau=t}^{t+t_0-1}\sum_{i=1}^LQ_i(\tau)g^{\tau-1}_{i}(\mu) 
+ \sum_{\tau=t}^{t+t_0-1}\sum_{j=1}^{M} H_j(\tau)(\dotp{h_j^{\tau-1}}{\mu} - b_j)
+V \sum_{\tau=t}^{t+t_0-1}(f^{\tau-1}(\mu^{\tau-1}) - f^{\tau-1}(\mu))\\
+\l(H^2 + G^2+2D_2^2 +  \frac{V^{2}}{2\alpha}D_1^2 +V\theta D_1+\alpha\theta\log d\r)t_0
+ \alpha\log(d/\theta).
\end{multline}
Furthermore, following the steps to obtain \eqref{eq:inter-bound-g} and \eqref{eq:inter-bound-h} by invoking $ \Vert \mu^{t} - \mu^{{t-1}} \Vert_{1} \leq 1$, we have
\begin{align*}
&\sum_{\tau=t}^{t+t_0-1}\sum_{j=1}^M\l(H_i(\tau)- H_j(t)\r)  \dotp{h_i^{\tau-1}}{\mu^\tau} 
\leq t_0^{2}H^2.\\
&\sum_{\tau=t}^{t+t_0-1}\sum_{i=1}^L\l(Q_i(\tau)- Q_i(t)\r)g^{\tau-1}_{i}(\mu) \leq t_0\l(\frac{3}{2G^2} + D_2^2\r), 
\end{align*}
and $V\sum_{\tau=t}^{t+t_0-1}f^{\tau-1}(\mu^{\tau-1}) \leq t_0VF$. Substituting these three bounds into \eqref{eq:drift-20} and recalling that $$\hat{C}_{V,\alpha,t_0}= 2\big(H^{2}+\frac{3}{2}G^{2}+D_{2}^{2}\big)t_{0}^{2}
 +2\big(H^2 + G^2+D_2^2 +  \frac{V^{2}}{2\alpha}D_1^2 +V\theta D_1+\alpha\theta\log d\big)t_0
+ 2\alpha \log(d/\theta)$$ gives the final bound.
\end{proof}

Using the previous bound, one can prove the following lemma:
\begin{lemma}\label{lem:q-bound-0}
If we take $V=\sqrt{T}, \alpha = T, t_{0} = T, \theta = 1/T$ in Algorithm \ref{alg:prob-simplex}, then the quantity  $\| \big(\mf Q(t),~\mf H(t)\big) \|_2$ satisfies the following conditions:
\begin{equation}
\expect{\Big\|\big(\mf Q(t),~\mf H(t)\big)\Big\|_2}
\leq  \hat{C}^{\prime} + \hat{C}^{\prime\prime} \sqrt{T} + \frac{2\log(d)}{\overline{c}} + \frac{2}{\overline{c}} \sqrt{T}\log{Td},
\end{equation}
where $\hat{C}^{\prime} = \frac{2}{\overline{c}}\Big( H^{2}+G^{2} + D_{2}^{2} +D_{1}^{2}/2 + D_{1}\Big)$ and $\hat{C}^{\prime\prime} = \frac{2}{\overline{c}}\Big( H^{2} + \frac{3}{2}G^{2} + D_{2}^{2} + F + \overline{l}(G+H+\overline{c}) + \overline{c}B + 2(2(G+D_{2})+H)^{2}\log(\frac{8(2(G+D_{2})+H)^{2}}{\overline{c}^{2}})\Big)$ are absolute constants independent of $d$ or $t$.
\end{lemma}
\begin{proof}[Proof of Lemma \ref{lem:q-bound-0}]
Following the same arguments as those in Lemma \ref{lem:dual-bound}, \ref{lem:bound-dual-function} and  $\ref{lem:q-bound}$, we can show 
\begin{multline*}
\expect{\Big\|\big(\mf Q(t),~\mf H(t)\big)\Big\|_2}
\leq 
\frac{\hat{C}_{V,\alpha,t_0}+
2\Big(F + \overline{\ell}(G+ H +\overline{c}) + \overline{c}B \Big)Vt_0}{\overline c t_0}\\
+ \frac{4t_0\Big(2 (G+D_2) 
+H\Big)^2}{\overline c}\log\l(\frac{8\l(2(G+D_2) + H\r)^2}{\overline c^2}\r).
\end{multline*}

Taking $V=\sqrt{T}, \alpha = T, t_{0} = T, \theta = 1/T$ and recalling the definition of $\hat{C}_{V, \alpha, t_{0}}$ yields
\begin{equation*}
\expect{\Big\|\big(\mf Q(t),~\mf H(t)\big)\Big\|_2}
\leq  \hat{C}^{\prime} + \hat{C}^{\prime\prime} \sqrt{T} + \frac{2\log(d)}{\overline{c}} + \frac{2}{\overline{c}} \sqrt{T}\log{Td},
\end{equation*}
where $\hat{C}^{\prime} = \frac{2}{\overline{c}}\Big( H^{2}+G^{2} + D_{2}^{2} +D_{1}^{2}/2 + D_{1}\Big)$ and $\hat{C}^{\prime\prime} = \frac{2}{\overline{c}}\Big( H^{2} + \frac{3}{2}G^{2} + D_{2}^{2} + F + \overline{l}(G+H+\overline{c}) + \overline{c}B + 2(2(G+D_{2})+H)^{2}\log(\frac{8(2(G+D_{2})+H)^{2}}{\overline{c}^{2}})\Big)$.
\end{proof}

The constraint violations in Theorem \ref{main:theorem-simplex} then follows by combining Lemma \ref{lem:simplex-constraint} and Lemma \ref{lem:q-bound-0}.

\subsection{Proof of other supporting lemmas}\label{sec:supp-lemma-proof}
\begin{proof}[Proof of Lemma \ref{cor:mfcq}]
We expand the simplex constraints in \eqref{eq:partial-static-prob} explicitly and the full dual function writes
\[
q_0^{(t,k)}(\lambda,~\eta,~\mathbf{u},~v): = 
\min_{\mu\in\mathbb{R}^d} 
\overline{f}^{t,k}(\mu) + \sum_{i=1}^L\lambda_i\nabla\overline{g}_i(\mu) 
+ \sum_{j=1}^M\eta_j\dotp{\expect{h_j^t}}{\mu} - \sum_{i=1}^du_i\mu_i + v(\sum_{i=1}^d\mu_i-1). 
\]
Let $q_0^* = \max_{\lambda\geq0,~\eta\in\mathbb{R}^M,~\mathbf{u}\geq0,v\in\mathbb{R}}~ q_0^{(t,k)}(\lambda,~\eta,~\mathbf{u},~v)$. 
By the assumption of lemma \ref{cor:mfcq} and Theorem \ref{theorem:mfcq} we have the solution set $K(\mu^*)$ of vectors $(\lambda,~\eta,~\mathbf{u},~v)$ of the following equations (KKT conditions) is non-empty and bounded:
\begin{align}
&\nabla\overline{f}^{t,k}(\mu^*) + \sum_{i=1}^L\lambda_i\nabla\overline{g}_i(\mu^*) 
+ \sum_{j=1}^M\eta_j\expect{h_j^t} - \sum_{i=1}^du_i\mathbf{e}_i + v\mathbf{1} = 0, \label{eq:stationary}\\
&\lambda\geq0,~~\mathbf{u}\geq0,  \nonumber\\
&\lambda_ig_i(\mu^*) = 0,~\forall i\in\{1,2,\cdots,M\},
~~u_i\mu_i^* = 0,~\forall i\in\{1,2,\cdots,d\}. \nonumber
\end{align}
It is easy to verify that $K(\mu^*) = \text{argmax}_{\lambda\geq0,~\eta\in\mathbb{R}^M,~\mathbf{u}\geq0,v\in\mathbb{R}} ~q_0^{(t,k)}(\lambda,~\eta,~\mathbf{u},~v)$ and we have zero duality gap, i.e. $q_0^* = \overline{f}^{(t,k)}(\mu^*)$. Our goal is to show that the set $\mathcal{V}^*$, defined in the statement of the lemma, is equal to the set 
$\{(\lambda^*,\eta^*)~|~(\lambda^*,\eta^*,\mf u^*,v^*)\in K(\mu^*),~\exists~\mf u^*,v^*\}$.

First of all, for any $(\lambda^*,\eta^*,\mf u^*,v^*)\in K(\mu^*)$, we have $q^{(t,k)}(\lambda^*,\eta^*)\geq q_0^{(t,k)}(\lambda^*,~\eta^*,~\mathbf{u}^*,~v^*)=q_0^*$. Since we have zero duality gap $q_0^* = \overline{f}^{(t,k)}(\mu^*)$ and one always has $q^{(t,k)}(\lambda,\eta)\leq \overline{f}^{(t,k)}(\mu^*),~\forall \lambda\in\mathbb{R}^L_+,~\eta\in\mathbb{R}^M$, it follows 
$q^{(t,k)}(\lambda^*,\eta^*) = \overline{f}^{(t,k)}(\mu^*)$. Thus, not only do we have a zero duality gap of $q^{(t,k)}(\lambda^*,\eta^*)$, we also have $\lambda^*,\eta^*$ being the solution point to the dual maximization problem $\text{max}_{\lambda\in\mathbb{R}^L_+,~\eta\in\mathbb{R}^M}q^{(t,k)}(\lambda,\eta)$, showing that $\mathcal{V}^*$ is non-empty and 
$\{(\lambda^*,\eta^*)~|~(\lambda^*,\eta^*,\mf u^*,v^*)\in K(\mu^*),~\exists~\mf u^*,v^*\}\subseteq\mathcal{V}^*$.

For the other direction, we pick any $(\lambda^*,\eta^*)\in\mathcal{V}^*$ and consider the following optimization problem:
\begin{equation}\label{eq:inter-opt-1}
q^{(t,k)}(\lambda^*,\eta^*) = \min_{\mu\in\Delta}~~
\overline{f}^{t,k}(\mu) + \sum_{i=1}^L\lambda_i^*\overline{g}_i(\mu) 
+ \sum_{j=1}^M\eta_j^*\overline{h}_j(\mu). 
\end{equation}
By zero duality gap, the solution to this optimization problem is equal to $\overline{f}^{(t,k)}(\mu^*)$. Thus $\mu^*$ must be one of the solution points of \eqref{eq:inter-opt-1} such that the complementary slackness $\lambda_i^*\overline{g}_i(\mu^*) = 0,~\forall i\in\{1,2,\cdots,L\}$ is satisfied.\footnote{Suppose on the contrary $\lambda_i^*\overline{g}_i(\mu^*) < 0$ for some index $i$, then, this means 
taking $\mu^*$ gives smaller value of the objective than $\overline{f}^{(t,k)}(\mu^*)$, contradicting the fact that the minimum is $\overline{f}^{(t,k)}(\mu^*)$.} Furthermore, it is obvious that MFCQ is also satisfied for \eqref{eq:inter-opt-1} (we only need to check the simplex constraints satisfy MFCQ, which is obvious). Thus, by Theorem \ref{theorem:mfcq}, we have there exists $\mf u^*\geq0, v^*\in\mathbb{R}$ such that the stationary condition \eqref{eq:stationary} is satisfied, and $u_i\mu_i^* = 0,~\forall i\in\{1,2,\cdots,d\}$. Combining with the previous complementary slackness $\lambda_i^*\overline{g}_i(\mu^*) = 0$, we arrive at the conclusion that $(\lambda^*,\eta^*,\mf u^*,v^*)\in K(\mu^*)$. This implies 
$\mathcal{V}^*\subseteq\{(\lambda^*,\eta^*)~|~(\lambda^*,\eta^*,\mf u^*,v^*)\in K(\mu^*),~\exists~\mf u^*,v^*\}$.
Overall, we have the set $\mathcal{V}^*$ is also bounded and we finish the proof.
\end{proof}

\begin{proof}[Proof of Lemma \ref{lem:ebc-estimate}]
First of all, note that by the EBC, for any $(\lambda,\eta)\in\mc S_\delta$, one has $\text{dist}((\lambda,\eta),\mathcal{V}^*)\leq C_\delta\delta^\beta$, thus, for those $(\lambda,\eta)$ such that $\text{dist}((\lambda,\eta),\mathcal{V}^*)\geq C_\delta\delta^\beta$,
$(\lambda,\eta)\not\in\mc S_\delta$. 
We then recall the following result:
\begin{lemma}[\cite{yang2015rsg}]\label{lem:yang-lin}
Consider any convex function $F:\mathcal{X}\rightarrow\mb R$ such that the minimal set $\Lambda^*$ is non-empty. Then, for any  $\mf x\in\mathcal{X}$ and any $\varepsilon>0$,
\[ 
\|\mf x- \mf x_\varepsilon^\dagger\|\leq \frac{\text{dist}(\mf x_\varepsilon^\dagger,\Lambda^*)}{\varepsilon}
\l( F(\mf x) - F(\mf x_\varepsilon^\dagger) \r),
\]
where $\mf x_\varepsilon^\dagger:= \text{argmin}_{\mf x_\varepsilon\in\mc S_\varepsilon}\|\mf x - \mf x_\varepsilon\|$, and 
$\mc S_\varepsilon$ is the $\varepsilon$-sublevel set defined in Lemma \ref{lem:ebc-estimate}.
\end{lemma}
Applying this lemma to our scenario, we define $(\lambda^\dagger_\delta,\eta^\dagger_\delta) = \argmin_{(\lambda_\delta,\eta_\delta)\in\mathcal{S}_\delta}
\|(\lambda_\delta,\eta_\delta) - (\lambda,\eta)\|_2$ and take function to be $q^{(t,k)}(\lambda,\eta)$ and consider the $\delta$-superlevel set $S_\delta$. By lemma \eqref{lem:yang-lin}, we readily have
\begin{align*}
\|(\lambda,\eta) - (\lambda_\delta^\dagger,\eta_\delta^\dagger)\|_2\leq &
\frac{\text{dist}((\lambda_\delta^\dagger,\eta_\delta^\dagger), \mathcal{V}^*)}{\delta}\l(
q^{(t,k)}(\lambda_\delta^\dagger,\eta_\delta^\dagger) - q^{(t,k)}(\lambda,\eta)\r)\\
\leq&
\frac{C_{\delta}\delta^{\beta}}{\delta}\l(q^{(t,k)}(\lambda_\delta^\dagger,\eta_\delta^\dagger) - q^{(t,k)}(\lambda,\eta) \r)\\
=&C_{\delta}\delta^{\beta-1}\l( q^{(t,k)}(\lambda_\delta^\dagger,\eta_\delta^\dagger) - q^{(t,k)}(\lambda,\eta)\r).
\end{align*}
On the other hand, 
\[
\text{dist}((\lambda_\delta^\dagger,\eta_\delta^\dagger),\mathcal{V}^* )
\leq C_{\delta}\l(q^{(t,k)}(\lambda^*,\eta^*) - q^{(t,k)}(\lambda_\delta^\dagger,\eta_\delta^\dagger)\r)^{\beta}
\]
Now, we claim that $q^{(t,k)}(\lambda^*,\eta^*) - q^{(t,k)}(\lambda_\delta^\dagger,\eta_\delta^\dagger) = \delta$. Indeed, suppose on the contrary, $q^{(t,k)}(\lambda^*,\eta^*) - q^{(t,k)}(\lambda_\delta^\dagger,\eta_\delta^\dagger) < \delta$, then, by the continuity of the function $q^{(t,k)}$, 
there exists $\alpha\in(0,1)$ and $(\lambda',\eta')=\alpha(\lambda^\dagger_\delta,\eta^\dagger_\delta)+(1-\alpha)(\lambda,\eta)$ such that $q^{(t,k)}(\lambda^*,\eta^*) - q^{(t,k)}(\lambda',\eta') = \delta$, i.e. $(\lambda',\eta')\in\mc S_{\delta}$, and $\|(\lambda,\eta)-(\lambda',\eta')\|_2 
= \alpha\|(\lambda,\eta)-(\lambda^\dagger_\delta,\eta^\dagger_\delta)\|_2 < \|(\lambda,\eta)-(\lambda^\dagger_\delta,\eta^\dagger_\delta)\|_2 $, contradicting the definition that 
$(\lambda^\dagger_\delta,\eta^\dagger_\delta) = \argmin_{(\lambda_\delta,\eta_\delta)\in\mathcal{S}_\delta}
\|(\lambda_\delta,\eta_\delta) - (\lambda,\eta)\|_2$.

Thus, we have
\[
\text{dist}((\lambda_\delta^\dagger,\eta_\delta^\dagger),\mathcal{V}^* )\leq C_{\delta}\delta^{\beta-1}\l(q^{(t,k)}(\lambda^*,\eta^*) - q^{(t,k)}(\lambda_\delta^\dagger,\eta_\delta^\dagger)\r).
\]
Overall, we have 
\[
\text{dist}((\lambda,\eta),\mathcal{V}^*) \leq \text{dist}((\lambda_\delta^\dagger,\eta_\delta^\dagger),\mathcal{V}^* ) + 
\|(\lambda,\eta) - (\lambda_\delta^\dagger,\eta_\delta^\dagger)\|_2
\leq C_{\delta}\delta^{\beta-1}\l( q^{(t,k)}(\lambda^*,\eta^*) - q^{(t,k)}(\lambda,\eta)\r),
\]
and we finish the proof.
\end{proof}


\end{document}